\documentclass{article}
\usepackage[utf8]{inputenc} \usepackage[T1]{fontenc}    \usepackage[english]{babel}

\usepackage{amsfonts}
\usepackage{nicefrac}       \usepackage{microtype}      \usepackage{lmodern}
\usepackage{amssymb,amsmath,amsthm}
\usepackage{stmaryrd}
\usepackage{bbm,bm}
\usepackage{latexsym}
\usepackage{xcolor}

\usepackage{enumerate}
\usepackage{verbatim}
\usepackage{booktabs}       

\usepackage{url}            \usepackage[colorlinks=true]{hyperref}
\usepackage[numbers,sort&compress,square,comma]{natbib}
\usepackage[capitalise,noabbrev]{cleveref}

\usepackage{parskip}
\usepackage{geometry}
\usepackage[short]{optidef}
\usepackage{algorithmic,algorithm}
\usepackage{authblk}

\usepackage{thmtools}

\declaretheorem{theorem}
\declaretheorem{corollary}
\declaretheorem{lemma}
\declaretheorem{proposition}
\declaretheorem{observation}
\declaretheorem{fact}

\declaretheoremstyle[qed=$\square$]{definitionwithend}
\declaretheorem[style=definitionwithend]{definition}
\declaretheorem[style=definitionwithend]{assumption}
\declaretheorem[style=definitionwithend]{example}
\declaretheorem[style=definitionwithend]{remark}

\crefname{fact}{Fact}{Facts}

\usepackage{import}
\usepackage{pdfpages}
\usepackage{transparent}

\definecolor{gold}{rgb}{0.85,0.65,0}

\newcommand{\abs}[1]{\ensuremath{\left\lvert #1 \right\rvert}}

\newcommand{\by}{\times}
 
\newcommand{\norm}[1]{\ensuremath{\left\lVert #1 \right\rVert}}
\newcommand{\ip}[1]{\ensuremath{\left\langle #1 \right\rangle}}

\let\emptyset\varnothing
\newcommand{\set}[1]{\left\{#1\right\}}

\newcommand{\bb}{\mathbb}

\def\L{{\mathbb{L}}}

\def\R{{\mathbb{R}}}
\def\S{{\mathbb{S}}}

\def\bS{{\mathbf{S}}}

\def\cB{{\cal B}}

\def\cE{{\cal E}}

\def\cL{{\cal L}}
\def\cM{{\cal M}}
\def\cN{{\cal N}}

\def\cP{{\cal P}}

\def\cR{{\cal R}}
\def\cS{{\cal S}}
\def\cT{{\cal T}}

\def\cY{{\cal Y}}

\DeclareMathOperator{\rank}{rank}
\DeclareMathOperator{\Diag}{Diag}

\DeclareMathOperator{\tr}{tr}

\DeclareMathOperator{\range}{range}

\DeclareMathOperator*{\E}{\mathbb{E}}

\DeclareMathOperator{\spann}{span}
\DeclareMathOperator{\inter}{int}

\DeclareMathOperator{\cl}{cl}
\DeclareMathOperator{\bd}{bd}

\DeclareMathOperator{\epi}{epi}

\DeclareMathOperator{\conv}{conv}
\DeclareMathOperator{\cone}{cone}
\DeclareMathOperator{\clconv}{clconv}
\DeclareMathOperator{\clcone}{clcone}
\def\extr{{\mathop{\rm extr}}}

 \usepackage{thm-restate}
\usepackage{tikz-cd}
\usepackage{tikz}
\usetikzlibrary{angles,quotes}
\DeclareMathOperator{\Sym}{Sym}

\graphicspath{{./}}

\PassOptionsToPackage{numbers, compress}{natbib}

\begin{document}

\title{Necessary and sufficient conditions for rank-one generated cones}
\author[1]{C.J.\ Argue}
\author[1]{Fatma K{\i}l{\i}n\c{c}-Karzan}
\author[1]{Alex L.\ Wang}
\affil[1]{Carnegie Mellon University, Pittsburgh, PA, 15213, USA.}
\date{\today}

\maketitle

\begin{abstract}

A closed convex conic subset $\cS$ of the positive semidefinite (PSD) cone is rank-one generated (ROG) if all of its extreme rays are generated by rank-one matrices. 
The ROG property of $\cS$ is closely related to the exactness of SDP relaxations of nonconvex quadratically constrained quadratic programs (QCQPs) related to $\cS$.
We consider the case where $\cS$ is obtained as the intersection of the PSD cone with finitely many homogeneous linear matrix inequalities and conic constraints and identify sufficient conditions that guarantee that $\cS$ is ROG.
Our general framework allows us to recover a number of well-known results from the literature. 
In the case of two linear matrix inequalities, we also establish the necessity of our sufficient conditions.
This extends one of the few settings from the literature---the case of one linear matrix inequality and the S-lemma---where an explicit characterization for the ROG property exists.
Finally, we show how our ROG results on cones can be translated into inhomogeneous SDP exactness results and convex hull descriptions in the original space of a QCQP.
We close with a few applications of these results; specifically, we recover the well-known perspective reformulation of a simple mixed-binary set via the ROG toolkit. \end{abstract}

\section{Introduction} \label{sec:introduction}

Let $\S^n$ denote the real vector space of $n\times n$  real symmetric matrices and $\S^n_+$ the cone of positive semidefinite matrices.
We will say that a closed convex cone $\cS\subseteq \S^n_+$ is \emph{rank-one generated} (ROG)\footnote{We will see in \cref{lem:rog_iff_extreme_rank_one} that the definitions of ROG cones given in the first sentence of the abstract and the second sentence of the main body are equivalent. For the purposes of our developments, we will begin with the definition given in the main body.} if
\begin{align*}
\cS =\conv(\cS \cap\set{xx^\top:\, x\in\R^n}),
\end{align*}
where $\conv(\cdot)$ is the convex hull operation.
In words, a closed convex cone $\cS$ is ROG if and only if it is equal to the convex hull of its rank-one matrices.

In most applications, the cone $\cS\subseteq\S^n_+$ will be represented as the intersection of $\S^n_+$ with a (possibly infinite) system of linear matrix inequalities (LMIs). Specifically, we will consider cones of the form
\begin{align*}
\cS(\cM) \coloneqq \set{X\in\S^n_+:\, \ip{M,X}\geq 0,\,\forall M\in\cM},
\end{align*}
where $\cM\subseteq\S^n$.
Note also that any closed convex cone $\cS\subseteq\S^n_+$ can be expressed in this form. An obvious question then is: What does the ROG property of $\cS(\cM)$ correspond to in terms of $\cM$, its defining LMIs?

While our main focus will be on closed convex cones, our results also have implications in the more general setting of arbitrary closed convex sets $\cS\subseteq\S^n_+$ and their defining LMIs.

\subsection{Motivation}
\label{subsec:motivation}
The ROG property is important in studying semidefinite program (SDP) relaxations of quadratically constrained quadratic programs (QCQPs).

QCQPs are a fundamental class of optimization problems that arise naturally in many areas. Indeed, many problems including binary integer linear programs, max-cut, max-clique, certain robust optimization problems  and polynomial optimization problems can be readily recast as QCQPs (see~\cite{phan1982quadratically,bao2011semidefinite,benTal2009book} and references therein).

It is well known that any QCQP can be reformulated as an SDP in a lifted space with an additional nonconvex rank constraint.
Dropping this rank constraint leads to the standard SDP relaxation~\cite{shor1990dual}. 
A general QCQP and its SDP relaxation are given by
\begin{align}
\label{eq:qcqp_sdp}
\inf_{y\in\R^{n-1}}\set{q_0(y):\,q_i(y)\geq 0,\,\forall i\in[m]} &= \inf_{x\in\R^n}\set{x^\top M_0 x:\,\begin{array}
	{l}
	x^\top M_i x \geq 0,\,\forall i\in[m]\\
	x_1^2 = 1
\end{array}}\nonumber\\
&\geq \inf_{X\in\S^n_+} \set{\ip{M_0, X}:\, \begin{array}
	{l}
	\ip{M_i, X}\geq 0,\,\forall i\in[m]\\
	X_{1,1} = 1
\end{array}}.
\end{align}
Here, $[m]\coloneqq\set{1,\ldots,m}$, the functions $q_i$ are quadratic
functions of the form $q_i(y) = y^\top A_i y + 2b_i^\top y + c_i$, the vector
$x$ should be thought of as $(\begin{smallmatrix}1 \\ y\end{smallmatrix})$, and
the matrices $M_i$ are defined as $M_i \coloneqq \left(\begin{smallmatrix}
	c_i & b_i^\top\\ b_i & A_i
\end{smallmatrix}\right)$.

In general, it is NP-hard to determine whether the SDP relaxation of a given QCQP is \textit{exact}, i.e., when equality holds in~\eqref{eq:qcqp_sdp} (see~\cite{laurent1995positive}). Nevertheless, sufficient conditions that ensure equality in~\eqref{eq:qcqp_sdp} are of great interest, and thus establishing such conditions has attracted a lot of attention in the literature.

Geometrically, SDP exactness occurs if and only if there exist rank-one matrices in the feasible domain of the SDP approaching its optimum value.
The ROG property is a similar but stronger notion of exactness. Specifically, if
the cone
\begin{align}
\label{eq:cone_cS_cM}
\cS(\set{M_1,\dots,M_m}) = \set{X\in\S^n_+:\, \begin{array}
	{l}
	\ip{M_i,X}\geq 0,\,\forall i\in[m]\\
\end{array}}
\end{align}
is ROG, then there exist rank-one matrices in the right hand side of 
\eqref{eq:qcqp_sdp}
approaching its optimum value \emph{for every} choice of $M_0$ such that the 
right hand side of \eqref{eq:qcqp_sdp} is finite. In other words, if the cone in 
\eqref{eq:cone_cS_cM}  is ROG, then equality holds in \eqref{eq:qcqp_sdp} \emph{for every} choice of objective function such that the SDP value is finite.
In the case of homogeneous QCQPs, i.e., where all $b_i=0$ and $c_i=0$ for $i=0,1,\ldots,m$, then \eqref{eq:cone_cS_cM}  is ROG if and only if the underlying SDP relaxation is exact for every choice of objective function. 
See \cref{subsec:QCQP_exactness} for a more detailed discussion of how equality
holding in \eqref{eq:qcqp_sdp} relates to the ROG property of $\cS(\set{M_1,\ldots,M_m)})$.

The ROG property is a natural strengthening of SDP exactness.
Consider, for example, the problem of minimizing an arbitrary quadratic function over an ellipsoid. The celebrated S-lemma \cite{yakubovich1971s} guarantees that the SDP relaxation of this problem is exact regardless of the choice of objective function.
One way of reinterpreting this statement is as the fact that
\begin{align*}
\cS(\set{M_1}) = 
\set{X\in\S^n_+:\, \begin{array}
	{l}
	\ip{M_1,X}\geq 0
\end{array}}
\end{align*}
is ROG when $M_1$ corresponds to an ellipsoid constraint.\footnote{Along with the observation that the SDP relaxation of this problem is always bounded.}
From a different perspective, the ROG property of spectrahedra can be thought of as an analogue of the \emph{integrality} property of polyhedra for linear programming relaxations of integer programs.
While there are well-known sufficient conditions such as total unimodularity or total dual integrality for the integrality property of polyhedra (see \cite{deCarliSilva2020notion} for recent developments and earlier references), 
the research on sufficient conditions for the ROG property of spectrahedra is much more recent and limited.

The ROG property is also relevant in the context of sum-of-squares (SOS) programming.
Consider a real homogeneous quadratic variety $V \coloneqq \set{x\in\R^n:\, x^\top M_i x = 0 ,\,\forall i\in[m]}$. Let $\cP_V$ denote the set of nonnegative quadratic forms on $V$, i.e., $\cP_V \coloneqq \set{M\in\S^n:\, x^\top M x \geq 0,\,\forall x\in V}$. Let $\Sigma_V$ denote the set of quadratic forms that are ``immediately nonnegative'' on $V$, i.e., $\Sigma_V \coloneqq \S^n_+ + \spann\set{M_i:\,i\in[m]}$, where $\spann(\cdot)$ is the span (linear hull) of the given elements.

It is clear that $\Sigma_V\subseteq\cP_V$.
A direct calculation shows that the dual cones of $\cP_V$ and $\Sigma_V$ are given by
\begin{align*}
\cP_V^* = \conv\set{xx^\top:\,\ip{M_i, xx^\top}= 0 ,\,\forall i\in[m]}
\quad \text{and} \quad
\Sigma^*_V= \set{X\in\S^n_+:\, \ip{M_i,X} = 0,\,\forall i\in[m]},
\end{align*}
respectively. Therefore, $\Sigma_V = \cP_V$ if and only if $\Sigma_V^*=\cP_V^*$, which holds if and only if $\Sigma_V^*$ is rank-one generated. In other words, every nonnegative quadratic form on $V$ is ``immediately nonnegative'' if and only if $\Sigma_V^*$ is ROG. 
See \cite[Section 6]{blekherman2017sums} for further connections and applications of the ROG property in the context of real algebraic geometry and statistics.

\subsection{Related literature}\label{sec:related}
\paragraph{Bounds on the rank of extreme points of general spectrahedra.}
A rich line of research has proven optimal worst-case bounds on the rank of 
extreme points of a spectrahedron (an affine slice of the PSD cone) in terms of the number of its defining linear matrix equalities (LMEs) \cite{barker1975cones,pataki1998rank}; see also \cite[Chapter II.13]{barvinok2002book}.
It is known that given $m$ LMEs, if there exists a positive semidefinite (PSD) solution to the LMEs, then there also exists a PSD solution with rank at most $r$ for any integral $r$ such that 
\begin{align*}
m < \binom{r+2}{2}.
\end{align*}
From this, we may deduce\footnote{After taking into account an additional LME due to the objective function and applying Strasziewicz Theorem (see \cite[Theorem 18.6]{rockafellar1970convex}).} that any spectrahedron defined by $m$ LMEs has only extreme points of rank at most $r$ for any integral $r$ satisfying $m+1 < \binom{r+2}{2}$.
In particular, taking $r = 1$, this bound implies that any spectrahedron defined by a single LME is ROG. Unfortunately, this bound does not shed much light onto (even the existence of) ROG spectrahedra in the case where $m>1$.
Although this bound is tight in general, it does not exploit potential structure in the defining LMEs. In other words, it is possible to achieve stronger bounds on the rank of extreme points of spectrahedra with additional structure. 
Our work complements this line of research by examining properties of systems of LMEs and LMIs that guarantee the ROG property beyond the case of $m = 1$.

\paragraph{SDP exactness.}
The question of when equality holds in \eqref{eq:qcqp_sdp} has attracted significant interest. Within this line of research, a number of papers study the classical trust region subproblem (TRS)---the problem of minimizing a nonconvex quadratic function over an ellipsoid---and its variants, and identify cases under which an exact SDP reformulation is possible.
This line of work can be traced back to Yakubovich's S-procedure \cite{yakubovich1971s,fradkov1979s-procedure} (also known as the S-lemma) and the work of \citet{sturm2003cones}.
We refer the interested readers to the excellent survey by \citet{burer2015gentle} and references therein.

It is worth noting that although the results in \cite{burer2015gentle} are stated in terms of the exactness of (strengthened) SDP relaxations, the underlying arguments in fact establish the ROG property for the corresponding SDP feasible domains.
For example, the domain of the SDP relaxation associated with the classical TRS is the intersection of $\S^n_+$ with a single LMI, which is well known to be ROG via S-lemma. In the other variants of TRS examined in \cite{burer2015gentle}, the domain of the associated exact SDP reformulation involves at least one problem specific conic constraint (in fact a second-order cone constraint), and consequently is described by an infinite family of well-structured LMIs.

These lines of work can be thought of as addressing the special case where there are only a few (usually one or two) nonconvex quadratic functions in the QCQP on the left of \eqref{eq:qcqp_sdp}. In contrast, \citet{burer2019exact} and \citet{wang2019tightness} recently introduced more general sufficient conditions for SDP exactness which do not make explicit assumptions on the number of nonconvex quadratic functions.
As an example, it can be shown that SDP exactness holds whenever a natural
symmetry parameter of the QCQP is large enough and the set of convex Lagrange (dual)
multipliers is polyhedral~\cite{wang2019tightness}.
See also~\cite{wang2020geometric} for sufficient conditions that make weaker assumptions on
the geometry of the set of convex Lagrange multipliers.
Some of these sufficient conditions for SDP exactness~\cite{wang2019tightness,wang2020geometric} have also been shown to guarantee that the (projection of the) epigraph of the SDP relaxation coincides exactly with the convex hull of the epigraph of the QCQP. In particular, the convex hulls of epigraphs of ``highly-symmetric'' QCQPs with favorable geometry are semidefinite-representable.
Results in this line of work generally depend heavily on how the objective function interacts with the constraints. Our work complements this line of research by establishing conditions for SDP exactness which are oblivious to the objective function.

\paragraph{Algebro-geometric properties of ROG spectrahedra.}
The ROG property has also been studied from a more algebro-geometric perspective \cite{hildebrand2016spectrahedral,blekherman2017sums}.

\citet{hildebrand2016spectrahedral} studies algebraic properties of ROG cones obtained by adding homogeneous LMEs to $\S^n_+$, and proves important facts about their representations.
The study begins by exploring the minimal defining polynomials and facial structure of ROG cones.
These properties are then used to build the main contribution of \cite{hildebrand2016spectrahedral}: The geometry of an ROG cone determines its representation as a linear section of a PSD cone (of any dimension) uniquely up to an isomorphism on the underlying vector space.
Additional results in this paper include a complete classification of ROG cones of degree\footnote{This is the degree of the minimal defining polynomial. This quantity is shown to be equivalent to the maximum rank over matrices in the ROG cone.} at most four as well as a number of operations on ROG cones (the direct product, full extension, and intertwining operations) that preserve the ROG property.

\citet{blekherman2017sums} study the ROG property of the cones $\Sigma_V^*$ (see \cref{subsec:motivation}) using techniques from real algebraic geometry and establish a connection between the geometry of $\Sigma_V^*$ and the property $N_{2,p}$ of the defining ideal of $V$.\footnote{A real projective variety $V$ satisfies property $N_{2,p}$ for an integer $p\geq 1$ if the $j$th syzygy module of the homogeneous ideal of $V$ is generated in degree at most $j+2$ for all $j<p$.}
Specifically, one of the main results in \cite{blekherman2017sums} is that, for general real projective varieties $V$, if $\Sigma_V^*$ has an extreme ray of rank $p>1$ then $V$ does not satisfy the property $N_{2,p}$.
This result is then strengthened in \cite[Theorem 20]{blekherman2017sums} to show that a spectrahedral cone $\cS$ defined by LMEs is ROG if and only if $\cS = \Sigma_V^*$ for a non-degenerate, reduced, 2-regular, totally real scheme $V$.
Finally, \cite{blekherman2017sums} also examines consequences of this connection to problems from real algebraic geometry, convex geometry, statistics, and real analysis, such as the positive semidefinite matrix completion problem.

In contrast to \cite{hildebrand2016spectrahedral,blekherman2017sums}, our results deal with possibly infinitely many linear matrix inequalities. The ROG property of such sets is not obvious and does not follow immediately from the ROG property of spectrahedral cones defined by LMEs. Indeed, we will see that both replacing equalities with inequalities (\cref{rem:chordal_nonexample}) and lifting inequalities to equalities (\cref{ex:lifting_rog_to_non_rog}) can destroy the ROG property of a spectrahedral cone. 
In addition, our more general setup allows us to handle additional interesting spectrahedral cones that have conic constraints, for example those arising from variants of the TRS.
We also discuss implications of the ROG property in terms of the exactness of SDP relaxations of QCQPs and explicit convex hull characterizations of sets defined by quadratic inequality constraints.
Finally, all of the proofs in this paper follow from elementary linear algebra and convex analysis. In particular, we hope that our results and their proofs shed light on the ROG property for readers less familiar with algebraic geometry.

\paragraph{ROG spectrahedra arising from PSD matrix completion.}
The ROG property has also been studied for spectrahedra arising in the matrix completion literature.
PSD matrix completion arises in a number of areas---for example in statistics, this problem is related to maximum likelihood estimation in Gaussian graphical models~\cite{dempster1972covariance}.
Let $E$ denote the edge set of an undirected graph on $n$ vertices that contains all self-loops. Let $K\subseteq\S^n$ denote the projection of $\S^n_+$ onto the indices in $E$. Then, a matrix $Y$ that is specified only on $E$ has a PSD completion if and only if it lies in the cone $K$. A short calculation shows that
\begin{align*}
K = \set{Y\in\S^n:\, \begin{array}
	{l}
	Y_{i,j} = 0,\,\forall (i,j)\notin E\\
	\ip{X,Y} \geq 0,\,\forall X\in \cS
\end{array}},
\qquad\text{where}\qquad
\cS = \set{X\in\S^n_+:\, X_{i,j} = 0,\,\forall(i,j)\notin E}.
\end{align*}
Consequently, the condition that every fully specified submatrix of $Y$ is positive semidefinite is necessary and sufficient for $Y$ to have a PSD completion if and only $\cS$ is ROG.
It is well-known that $\cS$ is ROG if and only if $E$ is the edge set of a chordal graph\footnote{A graph is chordal if every minimal cycle in the graph has at most 3 edges.} on $n$ vertices \cite{grone1984positive,agler1988positive,paulsen1989schur}.

\subsection{Overview and outline of the paper}

In this paper, we study necessary and/or sufficient conditions under which the intersection of the positive semidefinite cone with a set of homogeneous LMIs is an ROG cone. A summary of our contributions, along with an outline of the paper, is as follows:
\begin{enumerate}[(a)]
\item In \cref{sec:basic_tools}, we introduce our main terminology and basic tools.
Specifically, we show how the ROG property behaves when we switch from linear matrix inequalities (LMIs) to linear matrix equalities (LMEs) and how the ROG property for LMEs is characterized by the existence of solutions of quadratic systems.
In \cref{sec:knownbasicROGsets}, using our basic tools, we recover the well-known ROG set defined by a single LMI/LME (i.e., the S-lemma) and discuss a few implications for a simple sufficient condition in the case of two LMIs/LMEs. 

\item In \cref{sec:sufficient_conditions}, we establish a number of new sufficient conditions for the ROG property. As an example, we show that $\cS$ is ROG when $\cS = \set{X\in\S^n_+:\, Xc\in K}$ for a fixed vector $c$ and an arbitrary closed convex cone $K$. We also provide a number of examples to demonstrate that even simple extensions of our sufficient conditions are not possible.
We conclude this section by recovering the well-known result that the SDP relaxation strengthened with a second-order cone reformulation-linearization technique (SOC-RLT) inequality is exact for the variant of the TRS with a single linear inequality constraint.

\item A well-known consequence of the S-lemma is that the set $\cS(\cM)$ is ROG whenever $\cM = \set{M}$ is a single LMI; see e.g., \citet[Lemma 2.2]{ye2003new}.
In \cref{sec:necessary_conditions}, we give a complete characterization of ROG cones defined by two LMIs. 
One of our main results states a necessary and sufficient condition on the matrices $M_1$ and $M_2$ which ensures that the set $\cS$ is ROG. In particular, we establish in Theorem~\ref{thm:two_LMI_NS} that such a set is ROG if and only if the LMIs defined by $M_1$ and $M_2$ either ``only interact'' on a single face of $\S^n_+$ where they
 induce the same inequality constraint or both $M_1$ and $M_2$ have a specific indefinite rank-two structure.
We conclude that in the case of $m = 2$, there exist simple certificates of the ROG property.
\item In \cref{sec:rog_motivation}, we give a few applications of ROG cones.
In particular, we show how results on the ROG property of convex cones can be translated into inhomogeneous SDP exactness results and SDP-based convex hull descriptions of quadratically constrained sets. 
We then apply our ROG-based sufficient condition for exactness of the SDP relaxation to a simple set involving binary and continuous variables linked through a complementarity constraint. This gives a new method for deriving the well-known perspective reformulation for the convex hull of this set.
We close this section by presenting a number of examples that highlight how our ROG-based sufficient conditions for the SDP exactness and convex hull descriptions differ from other SDP exactness conditions in the literature.
The results in this section are self-contained and serve as additional motivation for the main study.
\end{enumerate}

We will compare our results with the literature in further detail in the sections as outlined above.

\subsection{Notation} \label{sub:notation}
For a positive integer $m$, let $[m]\coloneqq\set{1,\dots,m}$.
Let $\R_+^n$ denote the nonnegative vectors in $\R^n$. For $i\in[n]$, let $e_i\in\R^n$ denote the $i$th standard basis vector. Let $\bS^{n-1}$ denote the unit sphere in $\R^n$.
Let $\S^n$ denote the vector space of $n\times n$ real symmetric matrices and $\S^n_+$ the cone of positive semidefinite matrices.
We write $M\succeq 0$ (respectively $M\succ 0$) if $M$ is positive semidefinite (respectively positive definite).
For $M\in\R^{n\by n}$, let $\Sym(M)\coloneqq (M+M^\top) / 2\in\S^n$.
For $x\in\R^n$, let $\Diag(x)\in\S^n$ denote the diagonal matrix with $\Diag(x)_{i,i} = x_i$ for all $i\in[n]$. 
For a positive integer $n$, let $I_n$ denote the $n\times n$ identity
matrix. When the dimension is clear from context, we will simply write $I$ instead of $I_n$. 
For $M\in\S^n$, let $\range(M), \ker(M), \rank(M),\det(M),\tr(M)$ denote the range, kernel, rank, determinant and trace of $M$, respectively.
Let $\E$ denote an arbitrary Euclidean space.
Given a subset $\cM\subseteq\E$, let $\cl(\cM)$, $\inter(\cM)$, $\bd(\cM)$, $\conv(\cM)$, $\clconv(\cM)$, $\cone(\cM)$, $\clcone(\cM)$, and $\spann(\cM)$ denote the closure, interior, boundary, convex hull, closed convex hull, conic hull, closed conic hull, and span (linear hull) of $\cM$, respectively. For $\cM\subseteq\E$, let $\cM^\perp$ denote the subspace orthogonal to $\cM$.
For a subspace $W\subseteq \E$, let $\dim(W)$ denote its dimension. 
For a cone $K\subseteq\E$, let $\extr(K)$ denote its extreme rays and define $K^*\coloneqq\set{y\in\E:~\ip{x,y}\geq0,\,\forall x\in K}$ to be the dual cone of $K$. 
Given a subspace $W\subseteq \R^n$, 
we will identify $W$ with $\R^{\dim(W)}$. 
Let $\S^W$ denote $\S^{\dim(W)}$ identified with the linear subspace of $\S^n$ given by $\set{X\in\S^n:\, \range(X) \subseteq W}$.
For $x\in\R^n$, let $x_W\in W$ denote the projection of $x$ onto $W$.
For $M\in\S^n$, let $M_W\in\S^W$ denote the restriction of $M$ to $W$, i.e., $M_W \coloneqq U^\top MU$, where $U:W\to\R^n$ is the inclusion map.
When there is no confusion, let $0_{n}$ denote either the zero vector in $\R^n$ or the zero matrix in $\S^n$. Similarly, let $0_W$ denote either the zero vector in $W$ or the zero matrix in $\S^W$.
For $x\in W$ and $y \in W^\perp$, let $x\oplus y$ denote their direct sum.
For $X\in \S^W$ and $Y\in\S^{W^\perp}$, let $X\oplus Y$ denote their direct sum, i.e., the unique matrix in $\S^n$ such that $(x\oplus y)^\top (X\oplus Y)(x\oplus y) = x^\top X x + y^\top Y y$ for all $x\in W$ and $y\in W^\perp$.

\section{Properties of ROG cones}
\label{sec:basic_tools}

\subsection{Definitions}

Given $\cM\subseteq \S^n$, define
\begin{align*}
\cS(\cM)\coloneqq \set{X\in\S^n_+:\, \ip{M,X} \geq 0,\,\forall M\in\cM}.
\end{align*}
Note that $\cS(\cM)$ is a closed convex cone.
We are interested in the following property of such sets.
\begin{definition}\label{def:ROG}
A closed convex cone $\cS\subseteq \S^n_+$ is \emph{rank-one generated} (ROG) if
\begin{align*}
\cS &= \conv(\cS \cap \set{xx^\top:\, x\in\R^n}).\qedhere
\end{align*}
\end{definition}
\begin{remark}
Note that when $\cS\subseteq\S^n_+$ is a closed convex cone, we have $\conv(\cS\cap\set{xx^\top:\, x\in\R^n}) = \clconv(\cS\cap \set{xx^\top:\, x\in\R^n})$.\qedhere
\end{remark}

We will make extensive use of the following definitions and basic facts.
\begin{definition}
For $X\in\S^n$ nonzero, the ray spanned by $X$ is
\begin{align*}
\R_+ X \coloneqq \set{\alpha X:\, \alpha\geq 0}.
\end{align*}
Let $\cS\subseteq\S^n_+$ be a closed convex cone and suppose $X\in\cS$ is nonzero.
We say that $\R_+X$ is an extreme ray of $\cS$ if for any $Y,Z\in \cS$ such that $X = (Y+Z)/2$, we must have $Y,Z\in\R_+X$.\qedhere
\end{definition}

\begin{fact}
\label{fact:psd_modification_directions}
Let $X\in \S^n_+$. Then, $x\in \range(X)$ if and only if there exists $\epsilon>0$ such that $X-\epsilon xx^\top \in \S^n_+$.
\end{fact}

\begin{fact}
\label{fact:extreme_ray_def}
Let $\cS\subseteq\S^n_+$ be a closed convex cone. Then, for $X\neq0$, $\R_+ X$ is an extreme ray of $\cS$ if and only if for every $Y$,
\begin{align*}
[X-Y,X+Y]\subseteq \cS \implies \exists\, \alpha\in\R\text{ such that } Y=\alpha X.
\end{align*}
\end{fact}

The following fact follows immediately from \cref{fact:psd_modification_directions,fact:extreme_ray_def}.
\begin{fact}
\label{fact:rank_one_implies_extreme}
Let $\cS\subseteq\S^n_+$ be a closed convex cone. If $X\in\cS$ has $\rank(X) = 1$, then $\R_+X$ is an extreme ray of $\cS$.
\end{fact}

\begin{lemma}
\label{lem:rog_iff_extreme_rank_one}
Let $\cS\subseteq\S^n_+$ be a closed convex cone. Then, $\cS$ is ROG if and only if for each extreme ray $\R_+ X$ of $\cS$ we have $\rank(X)=1$.
\end{lemma}
\begin{proof}
$(\Leftarrow)$ Note that as $\cS$ is a subset of $\S^n_+$, it must be pointed. Then, as a closed convex pointed cone is the convex hull of its extreme rays, we have that $\cS= \conv(\cS \cap \set{xx^\top:\, x \in\R^n})$.

$(\Rightarrow)$ Let $\R_+X$ denote an extreme ray of $\cS$. As $\cS$ is ROG, we may by assumption write $X = \sum_{i=1}^k x_ix_i^\top$ where $x_ix_i^\top\in\cS$ for every $i\in[k]$. Then, as $\R_+X$ is an extreme ray of $\cS$, we must have $x_ix_i^\top\in \R_+X$ for every $i\in[k]$. Thus, we deduce that $X$ is rank-one.\qedhere

\end{proof}

The following fact allows us to decompose positive semidefinite matrices which
are identically zero on a given subspace.
\begin{lemma}
  \label{lem:block_decomp}
  Let $X\in\S^n_+$. Suppose $W\subseteq\R^n$ is a subspace on which $X_W = 0$. Then,
we can write $X = 0_W \oplus X_{W^\perp}$.
\end{lemma}
\begin{proof}
By performing an orthonormal change of variables, we may assume without loss of
generality that $W$
corresponds to the first $k$ coordinates of $\R^n$ and $W^\perp$ corresponds to
the last $n-k$ coordinates of $\R^n$. We can then write $X$ as a block matrix
\begin{align*}
  X = \begin{pmatrix}
    X_W & Y\\
    Y^\top & X_{W^\perp}
  \end{pmatrix}.
\end{align*}
Then, as $X\in\S^n_+$ and $X_W = 0$, we deduce that $Y = 0$. In particular, $X =
0_W \oplus X_{W^\perp}$.\qedhere
\end{proof}

\subsection{Relating LMIs to LMEs}\label{sub:relating_SM_and_TM}
Given a set $\cM\subseteq \S^n$, we will quickly switch from studying $\cS(\cM)$ to sets defined by LMEs, i.e., sets of the form
\begin{align*}
\cT(\cM) \coloneqq \set{X\in\S^n_+:\, \ip{M,X} = 0,\,\forall M\in\cM}.
\end{align*}
Sets of the form $\cT(\cM)$ are simpler to analyze than sets of the form $\cS(\cM)$.

\begin{remark}\label{rem:TM_finite_SM_infinite}
It is clear that given any $\cM\subseteq\S^n$, we have
$\cS(\cM) = \cS(\clcone(\cM))$ and $\cT(\cM) = \cT(\spann(\cM))$. In particular, we may without loss of generality assume that $\cM$ is finite when analyzing sets of the form $\cT(\cM)$---simply replace $\cM$ with a finite basis of $\spann(\cM)$. On the other hand, $\clcone(\cM)$  is not necessarily finitely generated.\qedhere
\end{remark}

We now present a series of lemmas relating $\cS(\cM)$ and $\cT(\cM)$ and their facial structures in terms of the ROG property. These results are particularly instrumental when we analyze the spectrahedral sets defined by finitely many LMIs/LMEs. 

\begin{lemma}
\label{lem:rog_iff_faces_rog}
For any set $\cM\subseteq \S^n$, the following are equivalent:
\begin{enumerate}
	\item $\cS(\cM)$ is ROG.
	\item Every face of $\cS(\cM)$ is ROG.
	\item $\cS(\cM)\cap \cT(\cM')$ is ROG for every $\cM'\subseteq\cM$.
\end{enumerate}
\end{lemma}
\begin{proof}
$(1. \Rightarrow 2.)$ 
Note that every extreme ray of a face of $\cS(\cM)$ is also an extreme ray of $\cS(\cM)$.

$(2. \Rightarrow 3.)$
First, suppose $\cM'=\emptyset$. Then, $\cT(\cM')=\S^n_+$ and thus $\cS(\cM)\cap \cT(\cM')=\cS(\cM)$. Since $\cS(\cM)$ is a face of itself, by part 2.\@ we deduce it is ROG.  Now consider any $\emptyset\neq\cM'\subseteq \cM$.
Note that $\cT(\cM')$ only depends on the linear span of $\cM'$, thus without loss of generality we may assume that $\cM'$ is a basis of $\spann(\cM')$. Take $Y$ to be the average of $\cM'$, i.e., $Y=\frac{1}{\abs{\cM'}}\sum_{M\in\cM'} M$. Note that $Y\in\cone(\cM')$ so that $Y\in\cS(\cM)^*$. We claim that $\cS(\cM)\cap\cT(\cM') = \cS(\cM) \cap Y^\perp$. Indeed, for all $X\in\cS(\cM)$, we have that $\ip{Y,X} = 0$ if and only if $\ip{Y,M} = 0$ for all $M\in\cM'$ if and only if $X\in\cT(\cM')$.
We deduce that $\cS(\cM)\cap\cT(\cM')$ is a face of $\cS(\cM)$, and thus it is ROG.

$(3.\Rightarrow 1.)$
Take $\cM' = \emptyset$.\qedhere
\end{proof}

We have the following immediate corollary of \cref{lem:rog_iff_faces_rog}.
\begin{corollary}
\label{cor:SM_rog_then_TM_rog}
For any set $\cM\subseteq \S^n$, if $\cS(\cM)$ is ROG then $\cT(\cM)$ is ROG.
\end{corollary}
\begin{proof}
Take $\cM'=\cM$ in \cref{lem:rog_iff_faces_rog}.\qedhere
\end{proof}

Informally, an extreme ray of $\cS(\cM)$ should also be an extreme ray of $\cS(\cM')$ for $\cM'\subseteq\cM$ as long as $\cM'$ contains the ``relevant'' inequalities in $\cM$. The following technical lemma makes this notion precise.
\begin{lemma}
\label{lem:compactness_extreme_for_tight}
Let $\cM\subseteq\S^n$ and let $\R_+X$ be an extreme ray of $\cS(\cM)$. Let $\cM'\subseteq\cM$ contain all of the constraints that are tight at $X$, i.e., $\set{M\in\cM:\, \ip{M,X} = 0}\subseteq\cM'$. If $\cM\setminus \cM'$ is compact, then $\R_+X$ is an extreme ray of $\cS(\cM')$. If additionally $\cM' = \set{M\in\cM:\, \ip{M,X} = 0}$, then $\R_+X$ is an extreme ray of $\cT(\cM')$.
\end{lemma}
\begin{proof}
Suppose $Y\in\S^n$ is such that $[X-Y,X+Y]\subseteq\cS(\cM')$. By compactness of $\cM\setminus \cM'$, we have that $\ip{M,X}$ achieves a positive minimum value on $\cM\setminus\cM'$. Furthermore, by compactness, $\ip{M,Y}$ is bounded on $\cM\setminus\cM'$. In particular, there exists $\epsilon>0$ small enough guaranteeing that $\ip{M,X\pm\epsilon Y}>0$ for all $M\in\cM\setminus\cM'$. This together with $[X-Y,X+Y]\subseteq\cS(\cM')$ implies that  $[X-\epsilon Y,X+\epsilon Y]\subseteq\cS(\cM)$. Thus, as $\R_+X$ is an extreme ray of $\cS(\cM)$ we conclude that $Y=\alpha X$ for some $\alpha\in\R$. This then implies that $\R_+X$ is extreme in $\cS(\cM')$.

The second statement follows by replacing $\cS(\cM')$ with $\cT(\cM')$ in the argument above.\qedhere
\end{proof}

\cref{lem:compactness_extreme_for_tight} allows us to strengthen \cref{lem:rog_iff_faces_rog} in a few ways.

\begin{lemma}
\label{lem:rog_iff_nontrivial_faces_rog}
Let $\cM\subseteq \S^n$ be compact. Then, $\cS(\cM)$ is ROG if and only if $\cS(\cM)\cap\cT(\cM')$ is ROG for every $\emptyset\neq \cM'\subseteq\cM$.
\end{lemma}
\begin{proof}
$(\Rightarrow)$ This direction follows \cref{lem:rog_iff_faces_rog}.

$(\Leftarrow)$ Let $\R_+ X$ be an extreme ray of $\cS(\cM)$ and define $\cM' \coloneqq \set{M\in\cM:\, \ip{M,X} = 0}$.
First suppose $\cM'\neq \emptyset$. As $\R_+ X$ is also an extreme ray of $\cS(\cM)\cap\cT(\cM')$, which by assumption is ROG, we have that $\rank(X) = 1$. Now suppose $\cM' = \emptyset$. By \cref{lem:compactness_extreme_for_tight} and the assumption that $\cM$ is compact, we deduce that $\R_+X$ is an extreme ray of $\cT(\emptyset) = \S^n_+$. We conclude that $\rank(X) = 1$.\qedhere
\end{proof}

We note that given \cref{lem:rog_iff_nontrivial_faces_rog}, it may be tempting to try to strengthen the third condition in \cref{lem:rog_iff_faces_rog} to the condition that $\cS(\cM)\cap \cT(\cM')$ is ROG for every $\emptyset\neq\cM'\subseteq\cM$. The following example shows that this is not possible without making the compactness assumption of \cref{lem:rog_iff_nontrivial_faces_rog}.
\begin{example}
Suppose $n = 2$ and $\cM =\bigcup_{i\in[4]} \cM_i$, where
\begin{equation*}
\begin{aligned}[c]
&\cM_1 = \set{\begin{pmatrix}
	1 &\\ & -1+\epsilon
\end{pmatrix}:\, \epsilon>0},
&&\cM_2 =
\set{\begin{pmatrix}
	-1 &\\ & 1+\epsilon
\end{pmatrix}:\, \epsilon>0},\\
&\cM_3 = \set{\begin{pmatrix}
	0 & 1\\ 1& \epsilon
\end{pmatrix}:\, \epsilon>0},
&&\cM_4 =
\set{\begin{pmatrix}
	0 & -1\\ -1& \epsilon
\end{pmatrix}:\, \epsilon>0}.
\end{aligned}
\end{equation*}

Noting that $\cS(\cM)$ is unchanged upon taking the closure of $\cM$ and that for all $i\in[4]$ and the constraints $\ip{M_{\epsilon},X}\geq0$ for $M_{\epsilon}\in\cM_i$ get only more restrictive as $\epsilon\to 0$, we deduce
\begin{align*}
\cS(\cM)= \cS\left(\set{\begin{pmatrix}
	1 & \\ & -1
\end{pmatrix}, \begin{pmatrix}
	-1 & \\ & 1
\end{pmatrix}, \begin{pmatrix}
	0 & 1\\ 1& 0
\end{pmatrix}, \begin{pmatrix}
	0 & -1 \\ -1 & 0
\end{pmatrix}}\right) = \R_+ I.
\end{align*}
We conclude $\cS(\cM) = \R_+I$ is not ROG. 
On the other hand, for any $\emptyset\neq \cM'\subseteq\cM$, we have $\cS(\cM)\cap\cT(\cM') = \set{0}$ (because $\ip{M, I}\neq 0$ for any $M\in\cM$) and is ROG.\qedhere
\end{example}

\begin{lemma}
\label{lem:TMprime_rog_then_SM_rog}
Let $\cM\subseteq \S^n$ be finite.
If $\cT(\cM')$ is ROG for every $\cM'\subseteq \cM$, then $\cS(\cM)$ is ROG.
\end{lemma}
\begin{proof}
Let $\R_+ X$ be an extreme ray of $\cS(\cM)$. Define $\cM' \coloneqq \set{M\in\cM:\, \ip{M,X} = 0}$.
By \cref{lem:compactness_extreme_for_tight} and the fact that any finite set is compact, we deduce that $\R_+ X$ is an extreme ray of $\cT(\cM')$. We conclude that $\rank(X) = 1$.\qedhere
\end{proof}

The following lemma shows that the ROG property of $\cT(\cM)$ is equivalent to
the ROG property of $\cT(\overline\cM)$ where $\overline\cM$ is the restriction
of $\cM$ onto the joint range of the matrices $M\in\cM$.
\begin{lemma}
\label{lem:TM_rog_iff_overline_TM_rog}
Let $W \coloneqq \spann\left(\bigcup_{M\in\cM} \range(M)\right)$. For $M\in\cM$, let $\overline M= M_W$ denote the restriction of $M$ to $W$. 
Let $\overline\cM = \set{\overline M:\, M\in\cM}$.
Then, $\cT(\cM)$ is ROG if and only if $\cT(\overline\cM)$ is ROG.
\end{lemma}
\begin{proof}
$(\Rightarrow)$
Note that $\cT(\overline\cM)$ is isomorphic to $\cT(\overline\cM)\oplus0_{W^\perp}$ via the rank-preserving map $X_W \mapsto X_W\oplus 0_{W^\perp}$.
We claim that $\cT(\overline\cM) \oplus 0_{W^\perp}$ is a face of $\cT(\cM)$. Indeed, we can write
\begin{align*}
\cT(\overline\cM) \oplus 0_{W^\perp} = \cT(\cM) \cap \set{X\in\S^n_+:\, \ip{0_W\oplus I_{W^\perp}, X} = 0}
\end{align*}
and note that $0_W\oplus I_{W^\perp}\in\S^n_+$. Then, $\cT(\overline\cM) \oplus 0_{W^\perp}$ is ROG by \cref{lem:rog_iff_faces_rog}. We conclude that $\cT(\overline \cM)$ is ROG.

$(\Leftarrow)$ Let $\R_+(X)$ be an extreme ray of $\cT(\cM)$ and set $\overline
X \coloneqq X_W$.
We will show that $\rank(X) = 1$ by considering two cases.
First, suppose $\overline X = 0$, then $\range(X)\subseteq W^\perp$. We deduce
that as $X\neq 0$, there exists a nonzero vector $y\in\range(X)\subseteq
W^\perp$. Note that $\ip{M, yy^\top} = \ip{M_W, (yy^\top)_{W}} = 0$.
Furthermore, $X\pm \epsilon yy^\top\in\S^n_+$ for all small enough $\epsilon>0$.
By the assumption that $\R_+(X)$ is an extreme ray, we then conclude that $X$ is
a scalar multiple of $yy^\top$ and is rank-one.

Next, suppose $\overline X\neq 0$.
As $\ip{M,X} = \ip{\overline M, \overline X}$ for every $M\in\cM$, we have that
$\overline X\in\cT(\overline\cM)$. By the assumption that $\cT(\overline\cM)$ is
ROG, we may write $\overline X = \sum_{i=1}^k \overline y_i\overline y_i^\top$
where $\overline y_i\overline y_i^\top\in
\cT(\overline\cM)$ are each nonzero.
Fix $\overline y \coloneqq \overline y_1$ and define $\overline z$ such that
$\overline y = \overline X \overline z$. This is possible as $\overline y
\in\range(\overline X)$. Finally, define
\begin{align*}
y \coloneqq X (\overline z \oplus 0_{W^\perp}).
\end{align*}
We claim that $X\pm \epsilon yy^\top \in\cT(\cM)$ for all $\epsilon>0$ small
enough. Indeed, as $y\in\range(X)$ we have that $X\pm \epsilon yy^\top
\in\S^n_+$ for all $\epsilon>0$ small enough. Furthermore, for all $M\in\cM$ we have
\begin{align*}
\ip{M, yy^\top} = \ip{\overline M, \overline y\overline y^\top} = 0,
\end{align*}
where the second equality follows from the fact that $\overline y \in
\cT(\overline \cM)$. Additionally note that $\overline y$ is nonzero and $y_W = \overline
y$ so that $y$ is nonzero.
We deduce that $X\pm \epsilon yy^\top \in \cT(\cM)$ for all $\epsilon>0$ small enough.
By the assumption that $\R_+(X)$ is an extreme ray, we then conclude that $X$ is
a scalar multiple of $yy^\top$ and is rank-one.\qedhere

\end{proof}

\begin{figure}
	\centering
		\begin{tikzcd}
		\boxed{\cM \text{ is finite and } \forall \cM'\subseteq \cM,\, \cT(\cM')\text{ ROG}} \arrow[r,Rightarrow]
		& \boxed{\cS(\cM)\text{ ROG}} \arrow[r,Rightarrow]
		& \boxed{\cT(\cM)\text{ ROG}}
		\end{tikzcd}
	\caption{A summary of \cref{lem:TMprime_rog_then_SM_rog,cor:SM_rog_then_TM_rog}}
	\label{fig:relating_S_and_T}
\end{figure}

\begin{remark}
  The characterizations given in \cref{lem:rog_iff_faces_rog,lem:rog_iff_nontrivial_faces_rog,cor:SM_rog_then_TM_rog,lem:TMprime_rog_then_SM_rog,lem:TM_rog_iff_overline_TM_rog,lem:compactness_extreme_for_tight} are based on the facial structure of the sets $\cS(\cM)$ and $\cT(\cM)$ and in a sense are analogous to characterizations of integral polyhedra.\qedhere
\end{remark}

\begin{remark}
The ROG property is not preserved under trivial liftings. 
When $\cM = \set{M_1,\dots,M_k}$ is finite, one may attempt to replace all of the inequalities defining $\cS(\cM)$ with equalities by adding new slack variables. Specifically, for $i\in[k]$, let $\overline{M}_i \in \bb S^{n+k}$ be the following block matrix
\begin{align*}
\overline{M}_i \coloneqq \begin{pmatrix}
	M_i &\\& e_ie_i^\top
\end{pmatrix}
\end{align*}
and let $\overline{\cM} \coloneqq \set{\overline{M}_1,\dots,\overline{M}_k}$.
It is straightforward to show that the ROG property is preserved under the projection of $\S^{n+k}$ onto $\S^n$. Thus, if $\cT\left(\overline{\cM}\right)$ is ROG, then $\cS(\cM)$ is also ROG. Unfortunately the reverse implication is not true in general. We will give a counterexample in \cref{sec:lifting_LMI_to_LME} (see \cref{ex:lifting_rog_to_non_rog}).\qedhere
\end{remark}

\subsection{Simple operations preserving ROG property}

We now present a few lemmas that are useful in reasoning about extreme rays of $\cS(\cM)$.
The following lemma states that an extreme ray $\R_+X$ ``only cares about'' constraints ``in the range of $X$.''
\begin{lemma}
\label{lem:extreme_for_constraints_in_range}
Let $\cM\subseteq\S^n$ and let $\R_+X$ be an extreme ray of $\cS(\cM)$. Let $W\coloneqq\range(X)$ and let $\cM_W\coloneqq \set{M_W:\, M\in\cM}$. Then $\R_+(X_W)$ is an extreme ray of $\cS(\cM_W)$. In particular, if $\cS(\cM_W)$ is ROG, then $\rank(X) = \rank(X_W) = 1$.
\end{lemma}
\begin{proof}
Suppose $Y_W\in\S^W$ is such that $[X_W-Y_W,X_W+Y_W]\subseteq\cS(\cM_W)$. Let $Y = 0_{W^\perp}\oplus Y_W$. Then, $X + Y = 0_{W^\perp}\oplus (X_W + Y_W)$, and for any $M\in\cM$ we have $\ip{M,X+Y} = \ip{M_W,X_W+Y_W}\geq 0$. We deduce that $X+Y\in\cS(\cM)$. Similarly $X-Y\in\cS(\cM)$ whence $[X-Y,X+Y]\subseteq\cS(\cM)$.
As $\R_+X$ is extreme in $\cS(\cM)$, we deduce that $Y = \alpha X$ for some $\alpha\in\R$.
Consequently, $Y_W = \alpha X_W$ for some $\alpha\in\R$ and $\R_+(X_W)$ is extreme in $\cS(\cM_W)$.\qedhere
\end{proof}

The following lemma addresses the case when $\cM$ can be partitioned into ``non-interacting'' sets of constraints.

\begin{lemma}\label{lem:non-interactingSufficientCond}
Let $\cM \subset \S^n$ be a finite union of compact sets $\cM = \bigcup_{i=1}^k \cM_i$.
Further, suppose that for all nonzero $X\in\S^n_+$ and $i\in[k]$, if $\ip{M_i,X} = 0$ for some $M_i\in\cM_i$, then $\ip{M, X}>0$ for all $M\in\cM\setminus \cM_i$.
Then, $\cS(\cM)$ is ROG if and only if $\cS(\cM_i)$ is ROG for all $i\in[k]$.
\end{lemma}
\begin{proof}
$(\Rightarrow)$
Fix $i\in[k]$ and let $\R_+ X$ be an extreme ray of $\cS(\cM_i)$.
If $\ip{M_i,X} >0$ for all $M_i\in\cM_i$, then \cref{lem:compactness_extreme_for_tight} implies that $\R_+ X$ is an extreme ray of $\S^n_+$ and so $\rank(X) = 1$. Now suppose $\ip{M_i,X} = 0$ for some $M_i\in\cM_i$. By assumption, $\ip{M,X}>0$ for all $M\in\cM\setminus \cM_i$ so that $X\in \cS(\cM)$. As $\cS(\cM)\subseteq\cS(\cM_i)$, we have that $\R_+X$ must also be an extreme ray of $\cS(\cM)$. We deduce that $\rank(X) = 1$.

$(\Leftarrow)$ Let $\R_+ X$ be an extreme ray of $\cS(\cM)$. Define $\cM' \coloneqq \set{M\in\cM:\, \ip{M,X}=0}$. If $\cM'=\emptyset$ then \cref{lem:compactness_extreme_for_tight} implies that $\R_+ X$ is an extreme ray of $\cT(\emptyset) = \S^n_+$ and so $\rank(X) = 1$.

Now suppose $\cM'$ is nonempty. Then, by assumption, $\cM'\subseteq\cM_i$ for some $i$. By \cref{lem:compactness_extreme_for_tight} and the assumption that $\cM\setminus\cM_i$ is compact, we deduce that $\R_+X$ is an extreme ray of $\cS(\cM_i)$. We conclude that $\rank(X) = 1$.\qedhere
\end{proof}

Finally, the following lemma states that an arbitrary intersection of ROG cones is ROG if and only if no new extreme rays are introduced.

\begin{lemma}
\label{lem:no_new_extreme}
Let $\cM\subseteq\S^n$ be a union $\cM = \bigcup_{\alpha\in A} \cM_\alpha$. Suppose that $\cS(\cM_\alpha)$ is ROG for every $\alpha\in A$.
Then, $\cS(\cM)$ is ROG if and only if
\begin{align*}
\extr(\cS(\cM)) \subseteq \bigcap_{\alpha\in A} \extr(\cS(\cM_\alpha)).
\end{align*}
\end{lemma}
\begin{proof}
$(\Leftarrow)$ Let $\R_+X$ be an extreme ray of $\cS(\cM)$. Then, by assumption, $\R_+X$ is an extreme ray of $\cS(\cM_\alpha)$ for each $\alpha\in A$. By recalling that each $\cS(\cM_\alpha)$ is ROG, we deduce $\rank(X) = 1$.

$(\Rightarrow)$ Let $\R_+X$ be an extreme ray of $\cS(\cM)$. Then, by the assumption that $\cS(\cM)$ is ROG, we have $\rank(X) = 1$. Next, note that $X\in\cS(\cM)=\bigcap_{\alpha\in A} \cS(\cM_\alpha)$, whence $X\in\cS(\cM_\alpha)$ for all $\alpha\in A$. Then as $\rank(X)=1$, we deduce that $\R_+X$ is extreme in $\cS(\cM_\alpha)$ for all $\alpha\in A$ by \cref{fact:rank_one_implies_extreme}.\qedhere
\end{proof}

\subsection{The ROG property and solutions of quadratic systems}

We next examine the ROG property of a set and its connection to the existence of
nonzero solutions of underlying quadratic systems of inequalities and/or equations.

\begin{definition}
Given $\cM\subseteq \S^n$ and $X\in\cS(\cM)$, we define
\begin{align*}
\cE(X,\cM) &\coloneqq \set{x\in\R^n:\,  \abs{x^\top Mx} \leq \ip{M,X},\,\forall M\in\cM}.\qedhere
\end{align*}
\end{definition}

\begin{lemma}
\label{lem:SM_rog_iff_nonzero_envelope}
$\cS(\cM)$ is ROG if and only if for every nonzero $X\in\cS(\cM)$ we have $\range(X)\cap \cE(X,\cM)\neq \set{0}$.
\end{lemma}
\begin{proof}
$(\Rightarrow)$ Suppose $X\in\cS(\cM)$ is nonzero. Because $\cS(\cM)$ is ROG, we can write $X = \sum_{i=1}^k x_ix_i^\top$ using nonzero matrices $x_ix_i^\top \in\cS(\cM)$. As $X$ is a nonzero matrix, we have $k\geq 1$ and thus $\bar x\coloneqq x_1$ exists.
Then, for every $M\in\cM$ and $i\in[k]$, we have $x_i^\top M x_i\geq 0$. In particular, $0\leq\bar x^\top M \bar x\leq \sum_{i=1}^k x_i^\top M x_i = \ip{M,X}$. Furthermore, $\bar x\in\range(X)$. We conclude that $\range(X)\cap \cE(X,\cM)$ contains the nonzero element $\bar x$.

$(\Leftarrow)$ Let $\R_+ X$ be an extreme ray of $\cS(\cM)$. By assumption, there exists a nonzero $x\in\range(X)$ such that
\begin{align*}
\abs{x^\top Mx} \leq \ip{M,X},\,\forall M\in\cM.
\end{align*}
By picking $\epsilon>0$ small enough, we can simultaneously ensure that $X\pm \epsilon xx^\top \in \S^n_+$ and that
\begin{align*}
\ip{M, X\pm \epsilon xx^\top}\geq (1-\epsilon) \ip{M,X}\geq 0 ,\,\forall M\in\cM.
\end{align*}
Hence, we conclude that the interval $[X-\epsilon xx^\top, X+\epsilon xx^\top]$ is contained in $\cS(\cM)$. In particular, because $\R_+ X$ is an extreme ray of $\cS(\cM)$, we deduce that $\epsilon xx^\top$ is a scalar multiple of $X$ and hence $\rank(X) = 1$.\qedhere
\end{proof}

When studying $\cT(\cM)$, we can replace the set $\cE(X,\cM)$ in \cref{lem:SM_rog_iff_nonzero_envelope} with a simpler set corresponding to solutions to a homogeneous system of quadratic equations.\footnote{Readers familiar with algebraic geometry will recognize this as the variety defined by $\cM$.}

\begin{definition}
Given $\cM\subseteq \S^n$, we define
\begin{align*}
\cN(\cM) &\coloneqq \set{x\in\R^n:\, x^\top M x = 0 ,\,\forall M\in\cM}.\qedhere
\end{align*}
\end{definition}

\begin{remark}\label{rem:NM_EXM_rel}
Note that for every $\cM\subseteq \S^n$ and every $X \in \cS(\cM)$, we have $\cN(\cM) \subseteq \cE(X,\cM)$.\qedhere
\end{remark}

\begin{corollary}
\label{cor:TM_rog_iff_nonzero_null}
$\cT(\cM)$ is ROG if and only if for every nonzero $X\in \cT(\cM)$ we have $\range(X)\cap \cN(\cM)\neq\set{0}$.
\end{corollary}
\begin{proof}
Note that $\cS(-\cM \cup \cM) = \cT(\cM)$ and apply \cref{lem:SM_rog_iff_nonzero_envelope}.\qedhere
\end{proof}

\begin{remark}
\label{rem:rank_at_least_two}
When applying \cref{lem:SM_rog_iff_nonzero_envelope}, it suffices to check the right hand side only for matrices $X$ with rank at least two. Indeed if $X=xx^\top$, then $x\in\range(X)\cap \cE(X,\cM)$. The same is true for \cref{cor:TM_rog_iff_nonzero_null}.\qedhere
\end{remark}

\begin{figure}
	\centering
	\begin{tikzcd}
		\boxed{\cS(\cM)\text{ ROG}} \arrow[r,Leftrightarrow]\arrow[d,Rightarrow]
		&\boxed{\forall X\in\cS(\cM)\setminus\set{0},\,\range(X)\cap\cE(X,\cM)\neq \set{0}} \arrow[d,Rightarrow]\\
		\boxed{\cT(\cM)\text{ ROG}} \arrow[r,Leftrightarrow]
		&\boxed{\forall X\in\cT(\cM)\setminus\set{0},\,\range(X)\cap\cN(\cM)\neq \set{0}}
		\end{tikzcd}
	\caption{A summary of \cref{lem:SM_rog_iff_nonzero_envelope} and \cref{cor:TM_rog_iff_nonzero_null}.}
\end{figure}

\subsection{Known ROG sets}\label{sec:knownbasicROGsets}
In order to familiarize the reader with our notation and setup, we now recover three known results in our language.
We begin with a result due to \citet{sturm2003cones} regarding spectrahedral cones defined by a single LMI.
\begin{lemma}
\label{lem:M_leq_1}
Consider any $M\in\S^n$, and let $\cM = \set{M}$. Then $\cS(\cM)$ is ROG.
\end{lemma}
\begin{proof}
By \cref{lem:rog_iff_nontrivial_faces_rog}, $\cS(\cM)$ is ROG if and only if $\cT(\cM)$ is ROG.
We will show that $\cT(\cM)$ is ROG by appealing to \cref{cor:TM_rog_iff_nonzero_null}.

Let $X\in \cT(\cM)$ have rank at least two.
Begin by performing a spectral decomposition $X = \sum_{i=1}^r \lambda_i x_ix_i^\top$, where $r=\rank(X)\geq 2$, the $x_i$ are orthonormal eigenvectors of $X$, and $\lambda_i>0$ for all $i\in[r]$.

If one of the eigenvectors $x_i$ is in $\cN(\cM)$, then $\range(X)\cap \cN(\cM)$ contains $x_i$ and is clearly nontrivial. 

Else, there exist distinct eigenvectors, without loss of generality $x_1$ and $x_2$, such that $\ip{M,x_1x_1^\top}>0>\ip{M,x_2x_2^\top}$. By continuity, there exists $x\in[x_1,x_2]$ such that $\ip{M,xx^\top}= 0$. Note that $x$ is nonzero as $0\notin [x_1,x_2]$ (this follows as $x_1$ and $x_2$ are orthonormal). Furthermore, $x\in\range(X)$. This concludes the proof as we have constructed a nonzero $x\in\range(X)\cap \cN(\cM)$.\qedhere
\end{proof}

Based on \cref{lem:TMprime_rog_then_SM_rog,lem:M_leq_1,cor:SM_rog_then_TM_rog}, we have the following characterization of ROG sets defined by two inequalities.
\begin{corollary}
\label{cor:M2_SM_rog_iff_TM_rog}
Suppose $\abs{\cM} = 2$, then $\cS(\cM)$ is ROG if and only if $\cT(\cM)$ is ROG.
\end{corollary}

The characterization given in \cref{cor:M2_SM_rog_iff_TM_rog} for the case of $\abs{\cM} = 2$ is, at the moment, unsatisfactory as we have yet to analyze when $\cT(\cM)$ is itself ROG.
Our developments in the remainder of this paper will make this implicit characterization much more explicit (see \cref{sec:necessary_conditions}).

Next, we recover a result related to the S-lemma~\cite{fradkov1979s-procedure} and a convexity theorem due to~\citet{dines1941mapping}.

\begin{lemma}\label{lem:psd_sum_rog}
Let $\cM = \set{M_1,M_2}$ and suppose there exists $(\alpha_1,\alpha_2)\neq (0,0)$ such that $\alpha_1M_1+\alpha_2M_2 \in \S^n_+$. Then, $\cS(\cM)$ is ROG.
\end{lemma}
\begin{proof}
By \cref{cor:M2_SM_rog_iff_TM_rog}, it suffices to show that $\cT(\cM)$ is ROG.
Recall also that $\cT(\cM)$ depends only on $\spann(\cM)$ (see \cref{rem:TM_finite_SM_infinite}), thus we may without loss of generality suppose $M_1 \in \S^n_+$.

Let $W\coloneqq \range(M_1)$. We claim that $X_W = 0$ for all $X\in\cT(\cM)$.
Indeed, suppose $X\in\cT(\cM)$ so that $\ip{M_1, X} = 0$. Noting that both
$M_1,X\in\S^n_+$, we deduce that $M_1 X = 0$ so that $X_W = 0$.
Then, applying \cref{lem:block_decomp} allows us to write $X = 0_W \oplus X_{W^\perp}$.

Let $\overline M_2 \coloneqq (M_2)_{W^\perp}$. Then,
\begin{align}
\label{eq:rank_preserving_isometry}
\cT(\cM) = \set{0_W \oplus X_{W^\perp}:\, \begin{array}
	{l}
	\ip{\overline M_2, X_{W^\perp}} = 0\\
	X_{W^\perp} \in \S^{W^\perp}_+
\end{array}} = 0_W \oplus \cT(\overline M_2).
\end{align}
By \cref{lem:M_leq_1} and \cref{cor:SM_rog_then_TM_rog}, $\cT(\overline M_2)$ is ROG. Then as $\cT(\cM)$ is isomorphic
to $\cT(\overline M_2)$ via the rank-preserving map $0_W \oplus X_{W^\perp}\mapsto X_{W^\perp}$, we conclude that $\cT(\cM)$ is ROG.\qedhere
\end{proof}

\begin{remark}
  \label{rem:geometric_interpretation_i}
  The condition that there exists $(\alpha_1,\alpha_2)\neq(0,0)$ such that
  $\alpha_1 M_1 + \alpha_2 M_2\in\S^n_+$ has a simple geometric interpretation.
  Specifically, this condition guarantees that the two LMEs defining
  $\cT(\set{M_1,M_2})$ only interact with each other on a single (possibly
  trivial) face of the positive semidefinite cone. Furthermore, on this face,
  the two LMEs impose the same (possibly trivial) constraint.\qedhere
\end{remark} 

\section{Sufficient conditions}
\label{sec:sufficient_conditions}

The following observation generalizes the key step in \cref{lem:psd_sum_rog}.
\begin{observation}
\label{obs:gen_psd_sum_rog_TM}
Let $\cM\subseteq\S^n$. Suppose there exists a nonzero $M\in\spann(\cM)\cap \S^n_+$. Let $W\coloneqq \range(M)$ and define $\cM_{W^\perp}\coloneqq\set{M_{W^\perp}:~M\in\cM}$. Then,
\begin{align*}
\cT(\cM) &= 0_W \oplus \cT(\cM_{W^\perp}).
\end{align*}
In particular, $\cT(\cM)$ is isomorphic to $\cT(\cM_{W^\perp})$ via the rank-preserving map $0_W \oplus Y\mapsto Y$ and $\cT(\cM)$ is ROG if and only if $\cT(\cM_{W^\perp})$ is ROG.
\end{observation}

\begin{remark}
\cref{obs:gen_psd_sum_rog_TM} simply notes that $\cT(\cM)$ is a subset of the
face $0_W\oplus \S^{W^\perp}_+$ of the positive semidefinite cone and then applies
\cref{lem:block_decomp}. This idea is linked to facial
reduction~\cite{borwein1981regularizing,liu2018exact,pataki2013strong},
a technique which has been used previously in the literature to
simplify semidefinite programs and more general conic programs.\qedhere
\end{remark}

Applying \cref{obs:gen_psd_sum_rog_TM} repeatedly gives the following generalization of \cref{lem:psd_sum_rog} as a sufficient condition for the ROG property.

\begin{proposition}
\label{prop:gen_psd_sum_rog_SM}
Let $\cM = \set{M_1,\dots,M_k}$ for some $k\geq 2$. Suppose for all distinct indices $i,j\in[k]$, there exists $(\alpha,\beta)\neq (0,0)$ such that $\alpha M_i + \beta M_j$ is positive semidefinite. Then, $\cS(\cM)$ is ROG.
\end{proposition}
\begin{proof}
By \cref{lem:TMprime_rog_then_SM_rog,lem:M_leq_1}, it suffices to show that $\cT(\cM')$ is ROG for every $\cM'\subseteq \cM$ with size at least two.

Let $\cM'\subseteq \cM$.
Consider repeatedly applying \cref{obs:gen_psd_sum_rog_TM} to get a chain of subspaces $W_1 \subset W_2 \subset \dots\subset W$ such that
\begin{align*}
\cT(\cM') = 0_{W_1} \oplus \cT(\cM'_{W_1^\perp}) = 0_{W_2} \oplus \cT(\cM'_{W_2^\perp}) = \dots = 0_W \oplus \cT(\cM'_{W^\perp}).
\end{align*}
We will repeat this process until $\spann(\cM'_{W^\perp})\cap \S^{W^\perp}_+ = \set{0}$. This process necessarily terminates as the subspaces $W_i$ strictly increase in dimension. Let $\overline M_i \coloneqq (M_i)_{W^\perp}$ and $\overline{\cM'}\coloneqq \set{\overline M_i:\, M_i \in \cM'}$.

We claim that $\dim(\spann(\overline{\cM'})) \leq 1$. Suppose otherwise and let $M_i, M_j\in\cM'$ such that $\overline M_i$ and $\overline M_j$ are independent. By assumption, there exists $(\alpha,\beta)\neq (0,0)$ such that $\alpha M_i+\beta M_j$ is positive semidefinite. Then,
\begin{align*}
\alpha \overline M_i + \beta \overline M_j = (\alpha M_i + \beta M_j)_{W^\perp}
\end{align*}
is positive semidefinite. Furthermore, this linear combination is nonzero by independence of $\overline M_i$ and $\overline M_j$. This contradicts the assumption that $\spann(\overline{\cM'})\cap\S^{W^\perp}_+ = \set{0}$.

Note that $\cT(\cM')$ is isomorphic to $\cT(\overline{\cM'})$ via the rank-preserving map $0_W \oplus X_{W^\perp} \mapsto X_{W^\perp}$. Furthermore, by \cref{rem:TM_finite_SM_infinite} and \cref{lem:M_leq_1}, we have that $\cT(\overline{\cM'})$ is ROG. We conclude that $\cT(\cM)$ is ROG.
\qedhere
\end{proof}
Intuitively, the conditions in this proposition have a similar geometric
interpretation to the conditions in \cref{lem:psd_sum_rog} (see
\cref{rem:geometric_interpretation_i}). Specifically, the proof shows that for any
$\cM'\subseteq\cM$ of size at least two, there exists a subspace
$W\subseteq\R^n$ such that $\cT(\cM')$ is contained in the face $0_W\oplus
\S^{W^\perp}_+$ of the positive semidefinite cone. Furthermore, on this
face, the LMEs in $\cM'$ all impose the same constraint.

Next, we present a new sufficient condition for the ROG property suggested by \cref{lem:SM_rog_iff_nonzero_envelope,rem:NM_EXM_rel}.

\begin{theorem}
\label{thm:ab_suffices}
Suppose $\cM = \set{\Sym(ab^\top):\, b\in \cB}$ for some $a\in\R^n$ and $\cB\subseteq\R^n$. Then, for every positive semidefinite $X$ of rank at least two, we have $\range(X)\cap\cN(\cM)\neq \set{0}$. In particular, $\cS(\cM)$ is ROG.
\end{theorem}
\begin{proof}
For any $v\in a^\perp$, we have $v^\top \Sym(ab^\top)v = v^\top ab^\top v = 0$. We deduce that $a^\perp\subseteq\cN(\cM)$, i.e., $\cN(\cM)$ contains a vector space of codimension one.

Let $X$ be a positive semidefinite matrix with rank at least two.
As $\dim(\range(X)) = \rank(X)$, we see that $\range(X)\cap \cN(\cM)$ must contain a vector space of dimension at least one. In particular, $\range(X)\cap \cE(X,\cM)\supseteq \range(X)\cap \cN(\cM)$ and is nonempty. Lemma~\ref{lem:SM_rog_iff_nonzero_envelope} then implies that $\cS(\cM)$ is ROG.\qedhere
\end{proof}

We list two immediate corollaries of Theorem~\ref{thm:ab_suffices}.

\begin{corollary}\label{cor:cone_c_suffices}
Let $K\subseteq \R^n$ be any closed convex cone and consider an arbitrary vector $c\in\R^n$. Then, the set $\set{X\in\bb S^n_+:~ Xc\in K}$ is ROG.
\end{corollary}
\begin{proof}
Define $\cM \coloneqq \set{\Sym(cb^\top):\, b\in K^*}$ where $K^*$ is the dual cone of $K$. Then $\set{X\in\S^n_+:\, Xc\in K} = \cS(\cM)$, whence \cref{thm:ab_suffices} implies the result.\qedhere
\end{proof}

\begin{corollary}\label{cor:ac_bc_suffices}
Let $a,b,c\in\R^n$. Then the set $\set{X\in\bb S^n_+:~ a^\top Xc\geq 0,\, b^\top X c\geq 0}$ is ROG.
\end{corollary}

By applying Lemma~\ref{lem:TMprime_rog_then_SM_rog} once more, we next give a sufficient condition which is not covered by Theorem~\ref{thm:ab_suffices}.

\begin{theorem}\label{thm:ab_ac_bc_suffices}
Let $a,b,c\in\R^n$. Then the set $\set{X\in\S^n_+:\, a^\top Xb \geq 0,\, b^\top Xc \geq 0,\,a^\top Xc\geq 0}$ is ROG.
\end{theorem}
\begin{proof}
Let $\cM = \set{\Sym(ab^\top), \Sym(ac^\top),\Sym(bc^\top)}$.
By Lemma~\ref{lem:TMprime_rog_then_SM_rog} and Corollary~\ref{cor:ac_bc_suffices}, it suffices to show that $\cT(\cM)$ is ROG.

We will show that $\cT(\cM)$ is ROG by appealing to Corollary~\ref{cor:TM_rog_iff_nonzero_null}.
Let $X\in \cT(\cM)$ have rank at least two.

Note that $\cN(\Sym(ab^\top))=a^\perp \cup b^\perp$. 
Hence, 
\[
\cN(\cM) 
=\left(a^\perp \cup b^\perp \right) \cap \left(a^\perp \cup c^\perp \right) \cap \left(b^\perp \cup c^\perp \right)
=\set{a,b}^\perp\cup\set{a,c}^\perp \cup\set{b,c}^\perp.
\]
If $Xa = Xb = Xc = 0$, then $\range(X)\subseteq \set{a,b,c}^\perp$ and thus $\range(X) \cap \cN(\cM) = \range(X)$ is clearly nontrivial.
Else, without loss of generality suppose $y = Xa \neq 0$. Because $X\in \cT(\cM)$, we have $b^\top y = c^\top y = 0$, and thus $y\in\cN(\cM)$. Noting that $y\neq0$ and $y\in\range(X)$, we have concluded $0\neq y \in \range(X) \cap \cN(\cM)$ as desired.\qedhere
\end{proof}
\begin{remark}
\label{rem:copositive_3}
By picking $n = 3$ and $\set{a,b,c} = \set{e_1,e_2,e_3}$ in
Theorem~\ref{thm:ab_ac_bc_suffices}, we recover the well-known fact that the set of doubly
nonnegative matrices (i.e., the set of matrices which are both entry-wise nonnegative and
positive semidefinite) in $\S^3$ is ROG.
In particular, this states that $X\in\S^3$ is doubly nonnegative if and only if it can be
written as $X = \sum_{i} x_ix_i^\top$ where $x_i\in\R^3$ are each entry-wise
nonnegative. In other words, the set of doubly nonnegative matrices and the set
of completely positive matrices in $\S^3$ coincide.
\qedhere
\end{remark}

\begin{remark}\label{rem:chordal_nonexample}
A graph $G=(V,E)$ is \emph{chordal} if every minimal cycle has at most 3 edges.
It is well-known that the set of positive semidefinite matrices with a fixed chordal support is ROG~\cite{agler1988positive,paulsen1989schur,grone1984positive}. Specifically, if $G = ([n], E)$ is a chordal graph containing all self-loops, then
\begin{align}
\label{eq:psd_support_set}
\set{X\in\S^n_+:\, X_{i,j} = 0 ,\,\forall (i,j)\notin E}
\end{align}
is ROG.

Unfortunately, the set in \eqref{eq:psd_support_set} does not necessarily remain ROG when the equality constraints are replaced with inequality constraints. Using our toolset, we illustrate this point below with an example. From this point of view, \cref{thm:ab_ac_bc_suffices} and \cref{rem:copositive_3} highlight a special chordal graph for which the inequality version of the set is also ROG.

Consider the path graph on four vertices with all self-loops. We will show that the following set is not ROG:
\begin{align*}
\cS = \set{X\in\S^4_+:\, \begin{array}
	{l}
	X_{1,2} \geq 0\\
	X_{2,3} \geq 0\\
	X_{3,4} \geq 0
\end{array}}.
\end{align*}
We will apply \cref{lem:SM_rog_iff_nonzero_envelope} to show that $\cS$ is not ROG.
Let $\cM = \set{\Sym(e_1e_2^\top),\Sym(e_2e_3^\top),\Sym(e_3e_4^\top)}$ so that $\cS = \cS(\cM)$.
Let $x = (1,\, 0,\, 1,\, 1)^\top$ and
$y = (0,\, 1,\, 1,\, -1)^\top$.
Note that the following rank-two matrix
\begin{align*}
X \coloneqq xx^\top + yy^\top = \begin{pmatrix}
	1 & 0 & 1 & 1\\
	0 & 1 & 1 & -1\\
	1 & 1 & 2 & 0\\
	1 & -1 & 0 & 2
\end{pmatrix}
\end{align*}
satisfies $X\in \cS$. We compute
\begin{align*}
\range(X) \cap \cE(X,\cM) &= \spann\set{x,y} \cap \set{z\in \R^4:\, \begin{array}
	{l}
	z_1z_2 = 0\\
	\abs{z_2z_3} \leq 1\\
	z_3z_4 = 0
\end{array}}.
\end{align*}
Let $z\in\range(X) \cap \cE(X,\cM)$. Then, writing $z= \alpha x + \beta y = (\alpha,\,\beta,\, \alpha+\beta,\, \alpha-\beta)^\top$, we deduce that $0=z_1z_2 = \alpha\beta $ and $0=z_3z_4=\alpha^2 - \beta^2$ so that $\alpha=\beta = 0$. Thus, $\range(X)\cap\cE(X,\cM)=\set{0}$.\qedhere
\end{remark}

Finally, we show how our results can be used to recover a result due to \citet{sturm2003cones}; see also \cite[Section 6.1]{burer2015gentle}. Let $\L^n\subseteq\R^n$ denote the second order cone (SOC)
\begin{align*}
\L^n\coloneqq \set{x=(y,t)\in\R^{n-1}\times\R:\, \norm{y}_2\leq t}.
\end{align*}
Defining $L \coloneqq \Diag(-1,\dots,-1,1)\in\S^n$, we can write $\L^n = \set{x\in\R^n:\, x^\top L x\geq 0,\, x_{n}\geq 0}$.

\begin{lemma}
\label{lem:one_cap}
Let $c\in\R^n$ and define
\begin{align*}
\cS \coloneqq \set{X\in\S^n_+:\, \begin{array}
	{l}
	Xc \in \L^n\\
	\ip{L,X} \geq 0
\end{array}}.
\end{align*}
Then, $\cS$ is ROG.
\end{lemma}
\begin{proof}
We begin by rewriting $\cS$ so that we may apply
\cref{lem:compactness_extreme_for_tight}. Let $\cB$ denote a compact base of $\L^n = (\L^n)^*$.  Then,
\begin{align*}
\cS = \cS\left(\set{L}\cup \set{\Sym(cb^\top):\, b\in\cB}\right).
\end{align*}

For the sake of contradiction suppose there exists an extreme ray $\R_+X$ of $\cS$ with $\rank(X)\geq 2$.

If $\ip{L,X}>0$ then $\R_+X$ is an extreme ray of
$\cS(\set{\Sym(cb^\top):\, b\in\cB}) = \set{X\in\S^n_+:\, Xc\in\L^n}$,
contradicting \cref{cor:cone_c_suffices}.
If $Xc\in\inter(\L^n)$ then $\R_+X$ is an extreme ray of
$\cS(\set{L}) = \set{X\in\S^n_+: \ip{X,L}\geq 0}$,
contradicting \cref{lem:M_leq_1}.
Finally, suppose $Xc = 0$ and let $W = \range(X)\subseteq c^\perp$.
Note that $X_W$ and $X$ have the same rank and
$\Sym(cb^\top)_{W} = 0$ for all $b\in\cB$.
Then, by \cref{lem:extreme_for_constraints_in_range}, we have that $\R_+(X_{W})$ is an extreme ray of $\cS(\set{L_{W}})$, contradicting \cref{lem:M_leq_1}.

In the remainder of the proof, we will assume that $\ip{L,X} = 0$ and $y\coloneqq Xc$ is a nonzero element in $\bd(\L^n)$, i.e., $y^\top L y = 0$.

Then, for all $\epsilon>0$ small enough, we have $X \pm \epsilon yy^\top\succeq 0$, $\ip{L,X\pm \epsilon yy^\top} = \ip{L,X}= 0$, and $(X\pm\epsilon yy^\top)c = (1\pm \epsilon y^\top c)y\in\L^n$. This contradicts the assumption that $\R_+ X$ is extreme. Thus, all extreme rays $\R_+X$ of $\cS$ have $\rank(X)\leq 1$.\qedhere
\end{proof}

\section{Necessary conditions} \label{sec:necessary_conditions}

In this section, we give a complete characterization of ROG cones defined by two LMIs.
\begin{theorem}
\label{thm:two_LMI_NS}
Let $\cM = \set{M_1,M_2}$. Then, $\cS(\cM)$ is ROG if and only if one of the following holds:
\begin{enumerate}[(i)]
	\item there exists $(\alpha_1,\alpha_2)\neq (0,0)$ such that $\alpha_1M_1+\alpha_2M_2 \in\S^n_+$, or
	\item there exists $a,b,c\in\R^n$ such that $M_1 = \Sym(ac^\top)$ and $M_2 = \Sym(bc^\top)$.
\end{enumerate}
\end{theorem}
Note that the \emph{if} direction of \cref{thm:two_LMI_NS} is a direct consequence of the sufficient conditions identified in \cref{prop:gen_psd_sum_rog_SM} and \cref{cor:ac_bc_suffices}. Furthermore, recall from \cref{cor:M2_SM_rog_iff_TM_rog} that when $\abs{\cM} =2$, the set $\cS(\cM)$ is ROG if and only if $\cT(\cM)$ is ROG. Thus, \cref{thm:two_LMI_NS} follows as a corollary to the following necessary condition.

\begin{theorem}
\label{thm:nec_conditions_two_LME}
Let $\cM = \set{M_1,M_2}$. If $\cT(\cM)$ is ROG, then one of the following holds:
\begin{enumerate}[(i)]
	\item there exists $(\alpha_1,\alpha_2)\neq (0,0)$ such that $\alpha_1M_1+\alpha_2M_2 \in\S^n_+$, or
	\item there exists $a,b,c\in\R^n$ such that $M_1 = \Sym(ac^\top)$ and $M_2 = \Sym(bc^\top)$.
\end{enumerate}
\end{theorem}

\begin{remark}
  \label{rem:gordan_stiemke}
  The conic Gordan--Stiemke Theorem (see Equation 2.3 in \cite{sturm2000error}
  and its surrounding comments) implies that for any subspace $W\subseteq
\S^n$,
\begin{align*}
W \cap \S^n_+ = \set{0} \iff W^\perp \cap \S^n_{++} \neq \emptyset.
\end{align*}
In particular, applying the conic Gordan--Stiemke Theorem in the context of 
\cref{thm:nec_conditions_two_LME} we deduce that
 if $M_1,M_2$ are linearly independent, then condition (i) in \cref{thm:nec_conditions_two_LME} 
 fails if and only if $\cT(\set{M_1,M_2})$ contains a positive definite matrix.\qedhere
\end{remark}

Conditions (i) and (ii) in \cref{thm:two_LMI_NS,thm:nec_conditions_two_LME} have
simple geometric interpretations.
See \cref{rem:geometric_interpretation_i} for a geometric interpretation of (i).
We describe an interpretation of condition (ii) in \cref{thm:nec_conditions_two_LME}, i.e., in the case of two LMEs.
Condition (ii) covers the important case when the two LMEs interact in a nontrivial manner inside $\S^n_+$. Suppose for the sake of presentation that $a = e_1$, $b=e_2$, $c=e_n$. Then, \cref{cor:ac_bc_suffices} implies that
\begin{align*}
\cT(\cM) &= \conv(\set{xx^\top:\, x_1x_n = 0,\, x_2x_n = 0})\\
&= \conv\left(\conv\set{xx^\top:\, x_1 = x_2 = 0}\cup \conv\set{xx^\top:\, x_n = 0}\right)\\
&= \conv\left((0_2\oplus \S^{n-2}_+ )\cup (\S^{n-1}_+\oplus 0_1) \right).
\end{align*}
In other words, condition (ii) covers the case where $\cT(\cM)$ is the convex hull of the union of two faces of the positive semidefinite cone with a particular intersection structure.
\cref{thm:nec_conditions_two_LME} states that these are the only ways for $\cT(\cM)$ to be ROG when $\abs{\cM} = 2$.

The proof of Theorem~\ref{thm:nec_conditions_two_LME} is nontrivial and will be the focus of the remainder of the section.
Before completing this proof, let us first work out in detail a prototypical example. This example will highlight a number of the steps of our proof.

\begin{example}
Suppose $\cM = \set{M_1,M_2}$ where $M_1 = \Diag(1,-1,0)$ and $M_2 = \Diag(0,1,-1)$ so that
\begin{align*}
\cT(\cM) = \set{X\in\S^3_+:\, X_{1,1} = X_{2,2} = X_{3,3}}.
\end{align*}
We first verify that neither condition (i) nor (ii) from \cref{thm:nec_conditions_two_LME} hold. Indeed, $\alpha_1 M_1 + \alpha_2 M_2 = \Diag(\alpha_1,\alpha_2-\alpha_1, -\alpha_2)$ is positive semidefinite if and only if $(\alpha_1,\alpha_2)=(0,0)$ so that condition (i) is violated. Next, note that $2M_1 +M_2 = \Diag(2,-1,1)$ has rank three so that condition (ii) is also violated. We next demonstrate that $\cT(\cM)$ is not ROG.

Let $w \coloneqq \left( 1, 1, \sqrt{2}\right)^\top$. We claim there exists a vector $z$ such that
\begin{align*}
\begin{pmatrix}
	z^\top M_1 z\\
	z^\top M_2z
\end{pmatrix}
= - \begin{pmatrix}
	w^\top M_1 w\\
	w^\top M_2 w
\end{pmatrix}.
\end{align*}
Indeed for this example, $z = \left( -1, 1, 0 \right)^\top$ is such a vector. It is clear that $w$ and $z$ are linearly independent so that $X\coloneqq ww^\top + zz^\top$ is a rank-two matrix contained in $\cT(\cM)$. By \cref{cor:TM_rog_iff_nonzero_null}, it suffices to show that $\range(X) \cap \cN(\cM) = \set{0}$. We will write a generic element from $\range(X)$ as
$\left( \alpha-\beta, \alpha+\beta, \sqrt{2}\alpha \right)^\top$. Then
\begin{align*}
\range(X) \cap\cN(\cM) &= \set{\begin{pmatrix}\alpha-\beta \\\alpha+\beta \\ \sqrt{2}\alpha\end{pmatrix}
	:\, \begin{array}
	{l}
	(\alpha-\beta)^2=(\alpha+\beta)^2=2\alpha^2
\end{array}}.
\end{align*}
The first equality implies $\alpha\beta = 0$. The second equality then implies that $\alpha=\beta=0$. We conclude $\range(X)\cap\cN(\cM) = \set{0}$ and that $\cT(\cM)$ is not ROG.\qedhere
\end{example}

We now begin on the proof of \cref{thm:nec_conditions_two_LME}. We first make a simplifying assumption that holds without loss of generality.

\begin{lemma}
\label{lem:M_span_n}
Let $W \coloneqq \spann\left(\bigcup_{M\in\cM} \range(M)\right)$. For $M\in\cM$, let $\overline M= M_W$ denote the restriction of $M$ to $W$. 
Let $\overline\cM = \set{\overline M:\, M\in\cM}$.
Then, $\cT(\cM)$ is ROG if and only $\cT(\overline\cM)$ is ROG.
Furthermore, if $\cM = \set{M_1,M_2}$ and $\overline \cM = \set{\overline M_1,\overline M_2}$, then each of conditions (i) and (ii) in \cref{thm:nec_conditions_two_LME} hold for $\cM$ if and only if they hold for $\overline \cM$.
\end{lemma}
\begin{proof}

The first part of this statement follows immediately from \cref{lem:TM_rog_iff_overline_TM_rog}.
The last statement of the lemma follows from definition of $W$.\qedhere
\end{proof}
We will henceforth assume that $\cM$ spans $\R^n$ in the following sense.
\begin{assumption}
\label{as:M_span_n}
Assume that $\spann\left(\bigcup_{M\in\cM} \range(M)\right)=\R^n$.\qedhere
\end{assumption}

\begin{proof}
[Proof of \cref{thm:nec_conditions_two_LME}.]
By \cref{lem:M_span_n}, we may without loss of generality assume that \cref{as:M_span_n} holds.
We will split the proof of \cref{thm:nec_conditions_two_LME} into a number of cases depending on the dimension $n$.
\begin{itemize}
	\item The case $n = 1$ holds vacuously as we can set $(\alpha_1,\alpha_2)$ to either $(1,0)$ or $(-1,0)$ to satisfy (i).
	\item For $n = 2$, we will suppose condition (i) is not satisfied and explicitly construct an extreme ray of $\cT(\cM)$ with rank two. The construction crucially uses the geometry of $\R^2$ (and $\S^2$). See \cref{prop:nec_n=2}.
	\item For $n=3$, we will suppose that neither conditions (i) nor (ii) are satisfied and explicitly construct extreme rays of $\cT(\cM)$ with rank two. The construction is based on understanding what the corresponding $\cN(\cM)$ set looks like. This construction crucially use the geometry of $\R^3$. See \cref{prop:nec_n=3}.
	\item Finally, we will show how to reduce the case of $n\geq 4$ to the case of $n = 3$. Specifically, supposing that $\cT(\cM)$ is a ROG cone, with $n\geq 4$, violating (i), we will construct $\overline\cM$ such that $\cT(\overline\cM)$ is a ROG cone, with $n = 3$, violating both (i) and (ii). See \cref{prop:nec_n_geq_4}.\qedhere
\end{itemize}
\end{proof}

\begin{remark}\label{rem:stronger_necessary_condition}
  Suppose \cref{as:M_span_n} holds.
  In this case, condition (ii) necessarily fails if $n\geq 4$. On the other hand if $n\leq 2$ and condition (ii) holds, then in fact condition (i) also holds.  In particular, condition (i) itself completely characterizes the ROG property of a cone defined by two LMIs whenever $n \neq 3$.

  Expanding \cref{as:M_span_n}, we have that condition (i) completely characterizes the ROG property of a cone defined by two LMIs whenever $\dim\left(\spann\left(\range(M_1)\cup\range(M_2)\right)\right) \neq 3$.\qedhere
\end{remark}

\begin{remark}\label{rem:certificates}
Both directions of \cref{thm:two_LMI_NS,thm:nec_conditions_two_LME} admit small certificates. 
\begin{itemize}
\item Suppose $\cS(\cM)$ is ROG. Then \cref{thm:two_LMI_NS} implies that there exists either aggregation weights $(\alpha_1,\alpha_2)\neq(0,0)$ for which $\alpha_1M_1 + \alpha_2M_2 \in\S^n_+$ or vectors $a,b,c\in\R^n$ for which $M_1 = \Sym(ac^\top)$ and $M_2 = \Sym(bc^\top)$.
\item Suppose $\cS(\cM)$ is not ROG. Then by \cref{thm:two_LMI_NS}, it suffices to certify that neither conditions (i) nor (ii) hold.
As $\cS(\cM)$ is not ROG, we may assume that $M_1$ and $M_2$ are linearly
independent. Then, the Gordan--Stiemke Theorem (see \cref{rem:gordan_stiemke}) implies that condition (i) fails if and only if there exists a positive definite matrix $X$ in $\cT(\cM)$. In other words, we can certify that condition (i) fails by presenting a positive definite matrix in $\cT(\cM)$.
If either $\rank(M_1)\geq 3$ or $\rank(M_2)\geq 3$, then the spectral decomposition of the corresponding $M_i$ certifies that condition (ii) does not hold.
Else, $M_1$ and $M_2$ are both indefinite rank-two matrices and we can write $M_1 = \eta_1 \Sym(ab^\top)$ and $M_2=\eta_2 \Sym(cd^\top)$ where $\eta_i \in\R$, $a,b,c,d\in\bS^{n-1}$. This decomposition is unique up to renaming $a$ and $b$ or $c$ and $d$. Then condition (ii) does not hold if and only if $a,b,c,d$ are distinct. In particular, this decomposition certifies that condition (ii) does not hold.\qedhere
\end{itemize}
\end{remark}

In the proof of \cref{thm:nec_conditions_two_LME}, we will make use of the following theorem related to the convexity of the joint image of two quadratic maps.

\begin{restatable}[\citet{dines1941mapping}]{theorem}{dinestheorem}\label{thm:dines}
Let $M_1,M_2\in\S^n$ and suppose that for all $(\alpha_1,\alpha_2)\neq (0,0)$, we have $\alpha_1M_1+\alpha_2M_2 \notin \S^n_+$. Then,
\begin{align*}
\set{\begin{pmatrix}
	x^\top M_1 x \\ x^\top M_2 x
\end{pmatrix}\in\R^2:\, x\in\R^n} = \R^2,
\end{align*}
i.e., for every $y\in\R^2$, there exists an $x\in\R^n$ such that $x^\top M_1x = y_1$ and $x^\top M_2 x = y_2$.
\end{restatable}

\subsection{Dimension $n=2$}
\label{subsec:necessary_2}

We now prove \cref{thm:nec_conditions_two_LME} for the case $n = 2$.

\begin{proposition}
	\label{prop:nec_n=2}
	Let $\cM = \set{M_1,M_2}$. Suppose \cref{as:M_span_n} holds and $n=2$. If $\cT(\cM)$ is ROG then there exists $(\alpha_1,\alpha_2)\neq (0,0)$ such that $\alpha_1M_1+\alpha_2M_2 \in\S^n_+$.
\end{proposition}

\begin{proof}
Suppose for all $(\alpha_1,\alpha_2)\neq(0,0)$, the linear combination $\alpha_1M_1  + \alpha_2M_2$ is not positive semidefinite.
In particular, $M_1$ and $M_2$ are linearly independent in $\S^2$.
Then, by Gordan--Stiemke Theorem (see \cref{rem:gordan_stiemke}), we deduce the existence of a positive definite matrix $X\in \cT(\cM)$.

Finally, as $\S^2$ has dimension three, the space orthogonal to both $M_1$ and $M_2$ has dimension one, so that in fact $\cT(\cM) = \R_+(X)$. We conclude that $\R_+(X)$ is an extreme ray with $\rank(X) = 2$.\qedhere

\end{proof}

\begin{figure}
	\begin{center}
		\begin{tikzpicture}
			\node[anchor=south] (image) at (0,0) {
				\includegraphics[width=0.25\textwidth]{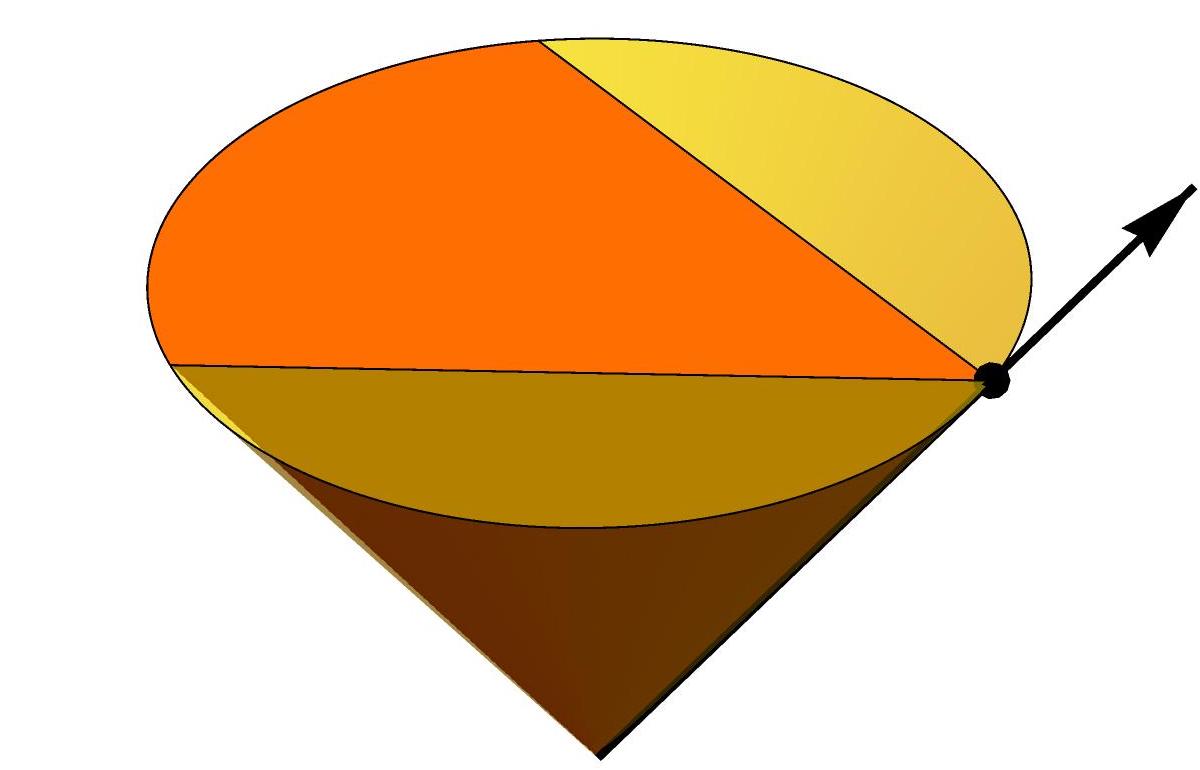}
			};
			\node[anchor=north west] at (1.6,1.8) {$\cT(\cM)$};
			\node[anchor=south] (image) at (5,0) {
				\includegraphics[width=0.2375\textwidth]{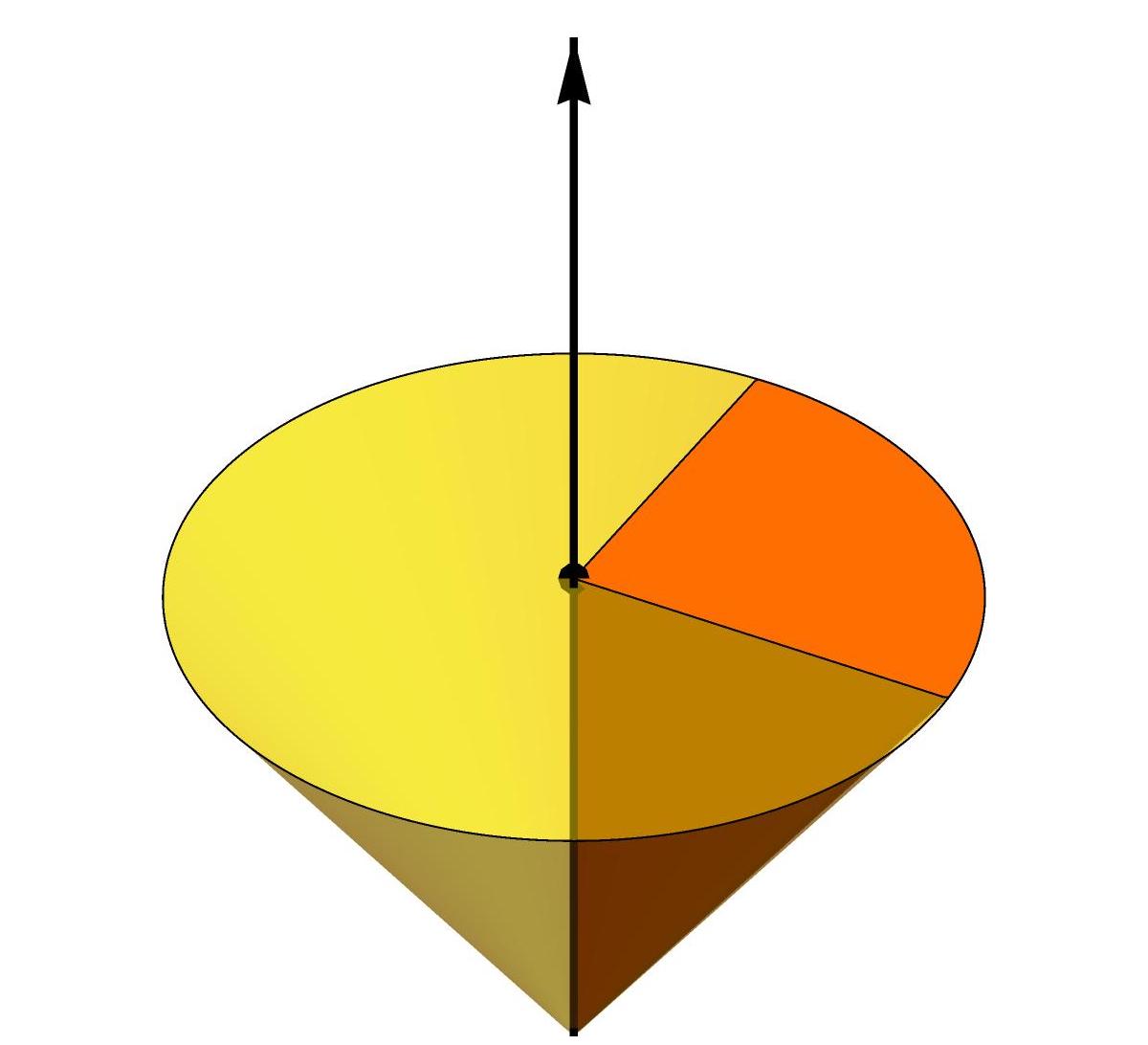}
			};
			\node[anchor=west] at (5,3) {$\cT(\cM)$};
		\end{tikzpicture}
	\end{center}
  \caption{For $n = 2$, every point on the interior of $\S^n_+$ has rank two and every point on the boundary of $\S^n_+$ has rank at most one. Condition (i) implies that $\cT(\cM)$, is either trivial or a ray in the boundary of $\cS^n_+$---this corresponds to the picture on the left. \cref{prop:nec_n=2} shows that when condition (i) is violated, $\cT(\cM)$ is a ray on the interior of $\cS^n_+$---this corresponds to the picture on the right.
  }
\end{figure}

\subsection{Dimension $n=3$}
\label{subsec:necessary_3}

We will make use of the following lemma from \citet[Lemma 3.13]{hildebrand2016spectrahedral}. The lemma states that the Carath\'eodory number of an element $X$ of $\cT(\cM)$ is equal to $\rank(X)$ when $\cT(\cM)$ is ROG.

\begin{lemma}[{\cite[Lemma 3.13]{hildebrand2016spectrahedral}}]
\label{lem:caratheodory}
Suppose $\cT(\cM)$ is ROG. For every $X\in \cT(\cM)$, we can write $X = \sum_{i=1}^{r} x_ix_i^\top$ where $x_i\in\cN(\cM)$ for all $i\in[r]$ and $r = \rank(X)$.
\end{lemma}

The next lemma states that when neither conditions (i) nor (ii) hold, the set $\cN(\cM)$ is extremely sparse in $\R^3$.

\begin{restatable}{lemma}{bezouttheorem}\label{lem:N_at_most_four_lines}
Let $\cM= \set{M_1,M_2}$. Suppose \cref{as:M_span_n} holds and $n=3$. If neither conditions (i) nor (ii) of \cref{thm:nec_conditions_two_LME} hold, then $\cN(\cM)$ is the union of at most four one-dimensional subspaces of $\R^3$.
\end{restatable}

Readers familiar with algebraic geometry will recognize this as a consequence of B\'ezout's theorem.\footnote{
Assuming that neither conditions (i) nor (ii) hold, the plane curves defined by $M_1$ and $M_2$ cannot share a common component. Then B\'ezout's theorem implies that $\cN(\cM)$ consists of at most four lines (or equivalently, four points in projective space).} For completeness, we provide an elementary proof of this lemma using only linear algebraic tools in \cref{ap:bezout_proof}.

We are now ready to prove \cref{thm:nec_conditions_two_LME} for the case of $n = 3$. We will assume that neither conditions (i) nor (ii) hold and use \cref{lem:N_at_most_four_lines} and \cref{thm:dines} to construct a rank-two matrix contained in $\cT(\cM)$. We will then apply \cref{lem:caratheodory} to derive a contradiction.

\begin{proposition}
	\label{prop:nec_n=3}
	Let $\cM = \set{M_1,M_2}$. Suppose \cref{as:M_span_n} holds and $n=3$. If $\cT(\cM)$ is ROG, then one of conditions (i) or (ii) of \cref{thm:nec_conditions_two_LME} must hold.
\end{proposition}
\begin{proof}
Suppose $\cT(\cM)$ is ROG but neither conditions (i) nor (ii) hold. Consider the subset of $\R^3$ given by
\begin{align*}
\cR \coloneqq \bigcup_{x,y\in\cN(\cM)} \spann(\set{x,y}).
\end{align*}
By \cref{lem:N_at_most_four_lines}, we have that $\cR$ is the union of a finite number of planes and lines in $\R^3$, and thus there exists $w\notin \cR$. By \cref{thm:dines}, we can pick $z$ such that
\begin{align*}
\begin{pmatrix}
	z^\top M_1 z\\
	z^\top M_2 z
\end{pmatrix} = 
- \begin{pmatrix}
	w^\top M_1 w\\
	w^\top M_2 w
\end{pmatrix}.
\end{align*}
As $w\notin\cR$, we deduce at least one of $w^\top M_1 w$ and $w^\top M_2 w$ is nonzero. Then, it is clear that $w$ and $z$ are linearly independent, and thus $X\coloneqq ww^\top + zz^\top$ is a rank-two matrix contained in $\cT(\cM)$.

As $\cT(\cM)$ is ROG, we can apply \cref{lem:caratheodory}. In particular, we can write $X = xx^\top + yy^\top$ for some $x,y\in\cN(\cM)$. Then, $w\in \range(X) = \spann(x,y)\subseteq \cR$. This contradicts our choice of $w\notin \cR$.\qedhere
\end{proof}

\subsection{Dimensions $n\geq 4$}
\label{subsec:necessary_geq_4}

We will now reduce the case of $n\geq 4$ to $n = 3$. The proof will show that if $\cM$ violates condition (i) then there exists a three-dimensional subspace $W$ for which the restriction of $\cM$ to $W$ fails both conditions (i) and (ii).

We begin by showing that there exists a linear combination of $M_1$ and $M_2$ with rank at least three.
\begin{lemma}
\label{lem:rank_three_lin_comb}
Let $\cM = \set{M_1,M_2}$. Suppose \cref{as:M_span_n} holds and $n\geq 4$. If condition (i) in \cref{thm:nec_conditions_two_LME} does not hold, then there exists $(\alpha_1,\alpha_2)$ such that $\rank(\alpha_1M_1+\alpha_2M_2)\geq 3$.
\end{lemma}
\begin{proof}
Suppose $\rank(\alpha_1M_1+\alpha_2M_2)\leq 2$ for all $(\alpha_1,\alpha_2)$. Because condition (i) does not hold, we conclude that for all $(\alpha_1,\alpha_2)\neq (0,0)$, the linear combination $\alpha_1M_1+\alpha_2M_2$ has exactly one positive and one negative eigenvalue.
Then, we can write $M_1 =\Sym(ab^\top)$ and $M_2 = \Sym(cd^\top)$. By \cref{as:M_span_n}, we have that $a,b,c,d$ are linearly independent. By independence, there exists an $x$ such that $x^\top b = 1$ and $x^\top a= x^\top c=x^\top d = 0$; we deduce that $(M_1+M_2)x = a\in\range(M_1+M_2)$. Similarly, $b,c,d\in\range(M_1+M_2)$. Then $\rank(M_1+M_2)=4$, a contradiction.\qedhere
\end{proof} 

We are now ready to prove \cref{thm:nec_conditions_two_LME} for the case of $n\geq 4$.
\begin{proposition}
\label{prop:nec_n_geq_4}
Let $\cM = \set{M_1,M_2}$. Suppose \cref{as:M_span_n} holds and $n\geq 4$. If $\cT(\cM)$ is ROG, then there exists $(\alpha_1,\alpha_2)\neq(0,0)$ such that $\alpha_1M_1+\alpha_2M_2\in\S^n_+$.
\end{proposition}

\begin{figure}
	\begin{center}
		\begin{tikzpicture}

			\begin{scope}
				\clip (-2.75,-2.75) rectangle (2.75,2.75);
				\fill[fill=gray,opacity=0.2]
					({10 * cos(30)}, {10 * sin(30)}) --
					({10 * cos(30) + 10 * cos(120)}, {10 * sin(30) + 10 * sin(120)}) --
					({-10 * cos(30) + 10 * cos(120)}, {-10 * sin(30)  + 10 * sin(120)}) --
					({-10 * cos(30)}, {-10 * sin(30)}) -- cycle;

				\fill[fill=gray,opacity=0.2]
					({10 * cos(-30)}, {10 * sin(-30)}) --
					({10 * cos(-30) + 10 * cos(240)}, {10 * sin(-30) + 10 * sin(240)}) --
					({-10 * cos(-30) + 10 * cos(240)}, {-10 * sin(-30)  + 10 * sin(240)}) --
					({-10 * cos(-30)}, {-10 * sin(-30)}) -- cycle;

				\fill[fill=blue,opacity=0.3]
					(0,-10) rectangle (10,10);
			\end{scope}

			\draw[blue, ->, thick] (0,0)--(1,0);
			\node at (1.85, 0) {$\left(\begin{smallmatrix}
				u_1^\top M_1 u_1\\
				u_1^\top M_2 u_1
			\end{smallmatrix}\right)$};
			\draw[blue, thick] (0,0.15) -- (0.15, 0.15) -- (0.15, 0);

			\draw[->] (0,0)--({cos(120)},{sin(120)});
			\node at ({1.75 * cos(120)},{1.75 * sin(120)}) {$\left(\begin{smallmatrix}
				u_2^\top M_1 u_2\\
				u_2^\top M_2 u_2
			\end{smallmatrix}\right)$};

			\draw[->] (0,0)--({cos(240)},{sin(240)});
			\node at ({1.75 * cos(240)},{1.75 * sin(240)}) {$\left(\begin{smallmatrix}
				u_3^\top M_1 u_3\\
				u_3^\top M_2 u_3
			\end{smallmatrix}\right)$};

			\begin{scope}
				\clip (-2.75 + 6.5,-2.75) rectangle (2.75 + 6.5,2.75);

				\fill[fill=gray,opacity=0.2]
					({10 * cos(120 - 2.5 - 90) + 6.5}, {10 * sin(120 - 2.5 - 90)}) --
					({10 * cos(120 - 2.5 - 90) + 10 * cos(120 - 2.5) + 6.5}, {10 * sin(120 - 2.5 - 90) + 10 * sin(120 - 2.5)}) --
					({-10 * cos(120 - 2.5 - 90) + 10 * cos(120 - 2.5) + 6.5}, {-10 * sin(120 - 2.5 - 90)  + 10 * sin(120 - 2.5)}) --
					({-10 * cos(120 - 2.5 - 90) + 6.5}, {-10 * sin(120 - 2.5 - 90)}) -- cycle;

				\fill[fill=gray,opacity=0.2]
					({10 * cos(240 - 10 - 90) + 6.5}, {10 * sin(240 - 10 - 90)}) --
					({10 * cos(240 - 10 - 90) + 10 * cos(240 - 10) + 6.5}, {10 * sin(240 - 10 - 90) + 10 * sin(240 - 10)}) --
					({-10 * cos(240 - 10 - 90) + 10 * cos(240 - 10) + 6.5}, {-10 * sin(240 - 10 - 90)  + 10 * sin(240 - 10)}) --
					({-10 * cos(240 - 10 - 90) + 6.5}, {-10 * sin(240 - 10 - 90)}) -- cycle;

				\fill[fill=blue,opacity=0.3]
					({10 * cos(5 - 90) + 6.5}, {10 * sin(5 - 90)}) --
					({10 * cos(5 - 90) + 10 * cos(5) + 6.5}, {10 * sin(5 - 90) + 10 * sin(5)}) --
					({-10 * cos(5 - 90) + 10 * cos(5) + 6.5}, {-10 * sin(5 - 90)  + 10 * sin(5)}) --
					({-10 * cos(5 - 90) + 6.5}, {-10 * sin(5 - 90)}) -- cycle;

			\end{scope}

			\draw[blue, ->, thick] (0 + 6.5,0)--({0.8 * cos(5) + 6.5},{0.8 * sin(5)});
			\node at ({1.85 * cos(5) + 6.5}, {1.85 * sin(5)}) {$\left(\begin{smallmatrix}
				x_1^\top M_1 x_1\\
				x_1^\top M_2 x_1
			\end{smallmatrix}\right)$};
			\draw[blue, thick] ({0.15 * cos(5) + 6.5}, {0.15 * sin(5)}) -- ({0.15 * cos(5) + 0.15 * cos(95) + 6.5}, {0.15 * sin(5) + 0.15 * sin(95)}) -- ({0.15 * cos(95) + 6.5}, {0.15 * sin(95)});

			\draw[->] (0 + 6.5,0)--({1.7 * cos(120 - 2.5) + 6.5},{1.7 * sin(120 - 2.5)});
			\node at ({2.3 * cos(120 - 2.5) + 6.5},{2.3 * sin(120 - 2.5)}) {$\left(\begin{smallmatrix}
				x_2^\top M_1 x_2\\
				x_2^\top M_2 x_2
			\end{smallmatrix}\right)$};

			\draw[->] (0 + 6.5,0)--({1.1 * cos(240 - 10) + 6.5},{1.1 * sin(240 - 10)});
			\node at ({1.85 * cos(240 - 10) + 6.5},{1.85 * sin(240 - 10)}) {$\left(\begin{smallmatrix}
				x_3^\top M_1 x_3\\
				x_3^\top M_2 x_3
			\end{smallmatrix}\right)$};
		\end{tikzpicture}
	\end{center}
  \caption{
The proof of \cref{prop:nec_n_geq_4} assumes that condition (i) in \cref{thm:nec_conditions_two_LME} does not hold for $\set{M_1,M_2}$ and constructs $u_1,u_2,u_3\in\R^n$ such that the vectors $\set{(u_i^\top M_1 u_i, u_i^\top M_2 u_i)}\subseteq\R^2$ are located as shown in the left figure. These vectors certify that condition (i) in \cref{thm:nec_conditions_two_LME} does not hold for $\set{M_1,M_2}$. Indeed, if $\alpha_1M_1 + \alpha_2M_2\in\S^n_+$, then $(\alpha_1,\alpha_2)$ must lie in the intersection of the three halfspaces defined by the $u_i$ vectors (one such halfspace is shaded in blue), whence $(\alpha_1,\alpha_2) = (0,0)$. The proof of \cref{prop:nec_n_geq_4} then observes that for all $x_1,x_2,x_3\in\R^n$ close enough to $u_1,u_2,u_3$, the vectors $\set{(x_i^\top M_1 x_i, x_i^\top M_2 x_i)}\subseteq\R^2$ certify that condition (i) in \cref{thm:nec_conditions_two_LME} also does not hold for $\set{(M_1)_W, (M_2)_W}$ where $W = \spann(\set{x_i})$. Again, the intersection of the corresponding halfspaces is trivial.}
  \label{fig:condition_i_certificate}
\end{figure}

\begin{proof}
Suppose for the sake of contradiction that $\cT(\cM)$ is ROG but condition (i) in \cref{thm:nec_conditions_two_LME} does not hold.

Let $\theta_1 \coloneqq 0$, $\theta_2 \coloneqq 2\pi/3$ and $\theta_3 \coloneqq 4\pi/3$.
Then, using \cref{thm:dines} we can find three vectors $u_1,u_2,u_3\in\R^n$ satisfying
\begin{align}
\label{eq:u_i_three_halfspaces}
\begin{pmatrix}
	u_i^\top M_1 u_i\\
	u_i^\top M_2 u_i
\end{pmatrix} = \begin{pmatrix}
	\cos(\theta_i)\\
	\sin(\theta_i)
\end{pmatrix} \qquad \forall i\in[3].
\end{align}
Note that $u_1, u_2, u_3$ certify that condition (i) does not hold for $\cM$ (see also \cref{fig:condition_i_certificate}):
\begin{align*}
  \set{(\alpha_1,\alpha_2):\, \alpha_1 M_1 + \alpha_2M_2 \succeq 0}&\subseteq \set{(\alpha_1,\alpha_2):\, u_i^\top (\alpha_1M_1 + \alpha_2M_2) u_i \geq 0 ,\,\forall i\in[3]}\\
                                                                   &=\set{(\alpha_1,\alpha_2):\, \ip{\begin{pmatrix}
                                                                         \alpha_1\\\alpha_2
                                                                       \end{pmatrix}, \begin{pmatrix}
                                                                         u_i^\top M_1 u_i\\
                                                                         u_i^\top M_2 u_i
                                                                       \end{pmatrix}}\geq 0,\,\forall i\in[3]}=\set{(0,0)}.
\end{align*}

Next, by \cref{lem:rank_three_lin_comb}, there exists $M_\beta\coloneqq
\beta_1M_1+\beta_2M_2$ with rank at least three. Let $v_1,v_2,v_3\in\R^n$ be
orthonormal eigenvectors of $M_\beta$ corresponding to nonzero eigenvalues.
Note that $v_1,v_2,v_3$ certify that condition (ii) does not hold for $\cM$:
\begin{align*}
\det\left( \begin{pmatrix}
    v_1^\top\\
    v_2^\top\\
    v_3^\top
  \end{pmatrix}M_\beta \begin{pmatrix}
    v_1 & v_2 & v_3
  \end{pmatrix} \right) \neq 0 \implies \rank(M_\beta)\geq 3.
\end{align*}

We will use the vectors $\set{u_i}$ and $\set{v_i}$ to construct a
three-dimensional subspace
$W\subseteq\R^n$ and show that the certificates of neither conditions (i) nor
(ii) holding in $\cM$ can be used to find certificates of neither conditions (i)
nor (ii) holding in $\set{(M_1)_W, (M_2)_W}$.

Let $\mu\in(0,1]$ to be fixed later.
Define $x_i \coloneqq (1-\mu) u_i + \mu v_i$ and set $W\coloneqq
\spann\set{x_1,x_2,x_3}$.
Let $\overline M_i \coloneqq (M_i)_W$ and set $\overline\cM \coloneqq
\set{\overline M_1,\overline M_2}$. Similarly define $\overline M_\beta$.

We first show that $W$ is a three-dimensional subspace for all $\mu>0$ small
enough. It is clear that $\dim(W)\leq 3$. To see that $\dim(W)\geq 3$ for all
$\mu>0$ small enough, consider the determinant of the orthogonal projections of
the $x_i$ vectors onto $\spann\set{v_1,v_2,v_3}$,
\begin{align*}
\det \left(\begin{pmatrix}
    v_1^\top\\
    v_2^\top\\
    v_3^\top
  \end{pmatrix} \begin{pmatrix}
  x_1 & x_2 & x_3
\end{pmatrix}\right) &=
                       \det \begin{pmatrix}
                         v_1^\top x_1 & v_1^\top x_2 & v_1^\top x_3\\
                         v_2^\top x_1 & v_2^\top x_2 & v_2^\top x_3\\
                         v_3^\top x_1 & v_3^\top x_2 & v_3^\top x_3
                       \end{pmatrix}.
\end{align*}
Recalling that the $x_i$s are each linear in $\mu$, we deduce that this
determinant is a degree-3 polynomial in $\mu$ which is not identically zero
(taking $\mu=1$ gives the determinant of the identity matrix), and thus
$\set{x_i}$ are linearly independent for all $\mu>0$ small enough.

Next, we show that condition (i) does not hold for $\overline\cM$ for all
$\mu>0$ small enough. Note that
\begin{align*}
\set{(\alpha_1,\alpha_2):\, \alpha_1 \overline M_1 + \alpha_2 \overline M_2 \succeq 0} &\subseteq \set{(\alpha_1,\alpha_2):\, x_i^\top \left(\alpha_1 M_1 + \alpha_2 M_2\right) x_i \geq 0 ,\,\forall i\in[3] }\\
&= \set{(\alpha_1,\alpha_2):\,\ip{\begin{pmatrix}
	\alpha_1\\
	\alpha_2
\end{pmatrix}, \begin{pmatrix}
	x_i^\top M_1 x_i\\ x_i^\top M_2 x_i
\end{pmatrix} }\geq 0,\,\forall i\in[3]},
\end{align*}
where the first relation follows from the definition of $\overline M_i$ and noting that $x_i\in W$. 
By continuity of the quadratic forms $x_i^\top M_1x_i$ and $x_i^\top M_2x_i$ in the variable $\mu$, and the choice of the $u_i$ in \cref{eq:u_i_three_halfspaces}, the set on the second line above is the trivial set $\set{0}$ for all $\mu>0$ small enough. Thus, $\overline\cM$ does not satisfy condition (i) for all $\mu>0$ small enough.

Next, we will show that $\overline M_\beta$ has rank three for all $\mu>0$ small enough.
Note that $\overline M_\beta$ is singular if and only if $\det(\overline
M_\beta) = 0$.
Picking the basis $\set{x_1,x_2,x_3}$ of $W$, we have that $\det(\overline
M_\beta) =0$ if and only if
\begin{align*}
  \det\left(\begin{pmatrix}
      x_1^\top\\
      x_2^\top\\
      x_3^\top
    \end{pmatrix} M_\beta \begin{pmatrix}
    x_1&
    x_2&
    x_3
  \end{pmatrix}\right) = 
\det\begin{pmatrix}
	x_1^\top M_\beta x_1 & x_1^\top M_\beta x_2 & x_1^\top M_\beta x_3\\
	x_2^\top M_\beta x_1 & x_2^\top M_\beta x_2 & x_2^\top M_\beta x_3\\
	x_3^\top M_\beta x_1 & x_3^\top M_\beta x_2 & x_3 ^\top M_\beta x_2\end{pmatrix} = 0.
\end{align*}
This is a degree-6 polynomial in $\mu$ (recall that $x_i$s are linear in $\mu$) that is not identically zero: for $\mu = 1$, this determinant evaluates to the product of three nonzero eigenvalues of $M_\beta$. Then, for all $\mu>0$ small enough, this polynomial is nonzero and hence $\rank(\overline M_\beta) = 3$. Thus, we deduce that  $\overline\cM$ does not satisfy condition (ii) for all $\mu>0$ small enough.

We now fix $\mu$ such that $\overline \cM$ does not satisfy either condition (i) or (ii). Note that this also fixes $W$.

To complete the proof we will show that $\cT(\overline \cM)$ is ROG. This will contradict \cref{prop:nec_n=3}.
Note that
\begin{align*}
 \cT(\overline{\cM}) \oplus 0_{W^\perp} = 
\cT(\cM) \cap \set{X\in\S^n_+:\,\ip{0_W \oplus I_{W^\perp}, X} = 0},
\end{align*}
which is a face of $\cT(\cM)$.
Then, as $\cT(\cM)$ is ROG, \cref{lem:rog_iff_faces_rog} implies that $\cT(\overline{\cM})\oplus 0_{W^\perp}$ is ROG. 
Next, note that $\cT(\overline{\cM})\oplus 0_{W^\perp}$ is isomorphic to $\cT(\overline{\cM})$ via the rank-preserving map $X_W \oplus 0_{W^\perp} \mapsto X_W$. We conclude that $\cT(\cM)$ is ROG.\qedhere
\end{proof}

\cref{prop:nec_n_geq_4}, together with \cref{prop:nec_n=2,prop:nec_n=3}, concludes the proof of \cref{thm:nec_conditions_two_LME}.

\subsection{Lifting LMIs into LMEs}\label{sec:lifting_LMI_to_LME}
In this section, we will show that a simple lifting of an LMI set $\cS$ into an LME set $\cT$ in a larger dimension may not preserve the ROG property.

\begin{example}
\label{ex:lifting_rog_to_non_rog}
Consider the set
\begin{align*}
\cS \coloneqq \set{X\in\S^3_+:\, \begin{array}
	{l}
	X_{1,2} = 0\\
	X_{1,3} \geq 0
\end{array}}.
\end{align*}
This set is ROG by \cref{thm:two_LMI_NS} and \cref{lem:rog_iff_faces_rog}.
We can replace the LMIs defining $\cS$ with LMEs in a lifted space as follows: Let $\Pi:\S^4\to\S^3$ denote the projection of a $4\by 4$ matrix onto its top-left $3\by 3$ principal submatrix. Then,
\begin{align*}
\cS = \Pi\left(\set{X\in\S^4:\, \begin{array}
	{l}
	X_{1,2} = 0\\
	X_{1,3} - X_{4,4} = 0
\end{array}}\right) = \Pi\left(\cT(\set{M'_1,M'_2})\right),
\end{align*}
where
\begin{align*}
M'_1 \coloneqq \begin{pmatrix}
	0 & 1/2 & 0 & 0\\
	1/2 & 0 & 0 & 0\\
	0 & 0 & 0 & 0\\
	0 & 0 & 0 & 0
\end{pmatrix}
\qquad\text{and}\qquad
M'_2 \coloneqq \begin{pmatrix}
	0 & 0 & 1/2 & 0\\
	0 & 0 & 0 & 0\\
	1/2 & 0 & 0 & 0\\
	0 & 0 & 0 & -1
\end{pmatrix}.
\end{align*}

Define $\cM'\coloneqq\set{M'_1,M'_2}$. By \cref{thm:two_LMI_NS}, we see that $\cT(\cM')$ is not ROG. We conclude that the obvious lifting of LMIs into LMEs can take ROG sets $\cS(\cM)$ to non-ROG sets $\cT(\cM')$ (even when there is only a single inequality to lift).\qedhere
\end{example}

\section{Applications of ROG cones}
\label{sec:rog_motivation}

\subsection{Exactness of SDP relaxations of QCQPs}
\label{subsec:QCQP_exactness}
In this subsection, we relate the ROG property of a cone $\cS$ to exactness results
for both homogeneous and inhomogeneous QCQPs and their relaxations.

The following lemma states that a cone $\cS\subseteq\S^n_+$ is ROG if and only if
the SDP relaxation of the corresponding homogeneous QCQP is exact for all
choices of objective function.

\begin{lemma}
Let $\cM\subseteq\S^n$. Then $\cS(\cM)$ is ROG if and only if for every
$M_0\in\S^n$,
\begin{align}
\label{eq:homogeneous_sdp_exactness}
\inf_{X\in \cS(\cM)}\ip{M_0,X} = \inf_{x\in\R^n}\set{\ip{M_0,xx^\top}:\, xx^\top\in\cS(\cM)}.
\end{align}
\end{lemma}
\begin{proof}
By \cref{def:ROG}, $\cS(\cM)$ is ROG if and only if $\cS(\cM)=\conv\left(\cS(\cM) \cap \set{xx^\top:\,
    x\in\R^n}\right)$. 
Moreover, both $\cS(\cM)$ and $\conv\left(\cS(\cM) \cap \set{xx^\top:\,
    x\in\R^n}\right)$ are closed convex cones so that they are equal if and
only if their dual cones are equal. Note that
\begin{align*}
M_0 \in \cS(\cM)^* \iff \inf_{X\in\cS(\cM)}\ip{M_0, X} = 0.
\end{align*}
Similarly,
\begin{align*}
  M_0 \in \left( \conv(\cS(\cM) \cap \set{xx^\top:\, x\in\R^n}) \right)^* &\iff
  \inf_{x\in\R^n} \set{\ip{M_0, xx^\top}:\, xx^\top\in\cS(\cM)}=0.
\end{align*}
Noting that both sides of \eqref{eq:homogeneous_sdp_exactness} can only
take the values $0$ or $-\infty$ completes the proof.\qedhere
\end{proof}

Next, we consider a general QCQP and its SDP relaxation.
Recall that in the general form given in \eqref{eq:qcqp_sdp}, a QCQP and its SDP
relaxation both contain exactly one inhomogeneous equality constraint.
The following lemma relates the ROG property of a cone to SDP exactness results
for its affine slices. This will allow us to apply our main results on spectrahedral cones to spectrahedra arising as the feasible domain of the SDP relaxations in \eqref{eq:qcqp_sdp}.

\begin{lemma}\label{lem:SDPtightness}
Let $\cM\subseteq\S^n$ and $B\in\S^n$.
If $\cS(\cM)$ is ROG, then
\begin{align*}
\inf_{x\in\R^n}\set{x^\top M_0 x:\, \begin{array}
	{l}
	x^\top M x \geq 0,\,\forall M\in\cM\\
	x^\top B x = 1
\end{array}}
= \inf_{X\in\S^n}\set{\ip{M_0, X}:\, \begin{array}
	{l}
	\ip{M, X}\geq 0,\,\forall M\in\cM\\
	\ip{B,X} = 1\\
	X\succeq 0
\end{array}}
\end{align*}
for all $M_0\in\S^n$ for which the optimum SDP objective value is bounded from below. In particular, this equality holds whenever the SDP feasible domain is bounded.
\end{lemma}
\begin{proof}
Let $\cS \coloneqq \cS(\cM)$.

$(\geq)$ This direction is immediate as the SDP gives a relaxation of the QCQP.

$(\leq)$ We may assume without loss of generality that the SDP is feasible. Let $X$ be a feasible SDP solution.
As $X\in\cS$ and $\cS$ is an ROG cone, there exist $x_1,\dots,x_r\in\R^n$ such that $x_ix_i^\top\in\cS$ for all $i\in[r]$ and $X = \sum_{i=1}^r x_ix_i^\top$. That is, we have $x_i^\top Mx_i \geq 0$ for all $M\in\cM$ and $i\in[r]$.
Without loss of generality, $x_i^\top B x_i$ is non-increasing in $i$ and there exists some $k\in[r]$ such that $x_1^\top Bx_1,\dots, x_k^\top B x_k$ are positive scalars summing to one.
Indeed, if this were to fail, we could first rearrange the indices in $[r]$ to get
$x_i^\top B x_i$ in non-increasing order and then subdivide the first term
$x_kx_k^\top$ for which $\sum_{i=1}^k x_i^\top B x_i \geq 1$
into two terms $\left(\sqrt{\alpha}x_k\right)\left(\sqrt{\alpha}x_k\right)^\top +
\left(\sqrt{1-\alpha}x_k\right)\left(\sqrt{1-\alpha}x_k\right)^\top$ (naturally,
also increasing $r$ to $r+1$)
so that the first $k$-many values of $x_i^\top B x_i$ are positive and sum to one. 
From here on we assume that such a transformation has been done (if needed), 
and $r$ reflects the final number of summands in this decomposition of $X$.

We may then write 
\begin{align*}
X = \hat X + \tilde X \coloneqq \left(\sum_{i=1}^k x_ix_i^\top\right) + \left(\sum_{i=k+1}^r x_ix_i^\top\right).
\end{align*}

Note that $\ip{B, \tilde X} = \ip{B, X} - \ip{B,\hat X} = 1-1=0$. Moreover, because the  optimum SDP objective value is bounded from below, we must have $\ip{M_0,\tilde X} \geq 0$.

For $i\in[k]$, define $\mu_i \coloneqq x_i^\top Bx_i>0$ and $\hat x_i \coloneqq x_i /\sqrt{\mu_i}$.
Then, $\hat x_i^\top B \hat x_i = 1$ and $\hat x_i ^\top M \hat x_i \geq 0$ for all $M\in\cM$ and $i\in[k]$. Finally, note that $1= \sum_{i=1}^k x_i^\top B x_i = \sum_{i=1}^k \mu_i$. Using these facts, we deduce
\begin{align*}
&\ip{M_0,X}
\geq \ip{M_0,\hat X}
= \sum_{i=1}^k x_i^\top M_0 x_i
= \sum_{i=1}^k \mu_i \hat x_i^\top M_0 \hat x_i\\
&\qquad\geq \min_{i\in[k]}\hat x_i^\top M_0 \hat x_i
\geq \inf_{x\in\R^n}\set{x^\top M_0 x:\, \begin{array}
	{l}
	x^\top M x \geq 0,\,\forall M\in\cM\\
	x^\top B x = 1
\end{array}}.
\end{align*}
The desired result follows by taking the infimum of this inequality over feasible solutions $X$ to the SDP.\qedhere
\end{proof}
\begin{remark}
\cref{lem:SDPtightness} extends \cite[Lemma 1.2]{hildebrand2016spectrahedral}, which shows that the same statement holds in the case of finitely many LMEs. The proof we present is new and immediately shows how to construct a QCQP feasible solution achieving the SDP value (or a sequence approaching the SDP value).\qedhere
\end{remark}

\begin{example}
  \label{ex:sdp_exactness_not_ROG}
  The reverse implication in \cref{lem:SDPtightness} is not true in general. In
  particular, consider the following example. Let
  \begin{align*}
    \cS = \set{\begin{pmatrix}
        \alpha & & \\ & \beta & \\ && \beta
      \end{pmatrix}:\, \alpha,\beta\geq 0} \subseteq\S^3_+,
  \end{align*}
  and set $B = e_1e_1^\top$. Note that $\cS$ has a rank-two extreme ray and thus is not ROG. Let
  $M_0 \in\S^3$. 
  A short calculation shows that the SDP relaxation of the QCQP defined by $M_0$ and $\cM$ associated with $\cS$ satisfies
  \begin{align*}
    \inf_{X\in\S^3}\set{\ip{M_0, X}:\, \begin{array}
                                         {l}
                                         X \in \cS\\
                                         X_{1,1} = 1
                                       \end{array}}
    = \begin{cases}
      (M_0)_{1,1} & \text{if } (M_0)_{2,2} + (M_0)_{3,3} \geq 0,\\
      -\infty & \text{else}.
    \end{cases}
  \end{align*}
  In particular, if $M_0\in\S^3$ is such that the optimum value of the SDP
  relaxation is bounded below, then the SDP relaxation takes the value
  $(M_0)_{1,1}$. On the other hand, $e_1e_1^\top\in \cS$ is a rank-one matrix
  achieving the same objective value. We deduce that
  \begin{align*}
    \inf_{x\in\R^3}\set{x^\top M_0 x:\ \begin{array}
                                         {l}
                                         xx^\top \in\cS\\
                                         (xx^\top)_{1,1} = 1
                                       \end{array}}
    =
    \inf_{X\in\S^3} \set{\ip{M_0, X}:\ \begin{array}
                                         {l}
                                         X\in\cS\\
                                         X_{1,1} = 1
                                       \end{array}}
  \end{align*}
  for all $M_0\in\S^3$ for which the right hand side is bounded below. \qedhere
\end{example}

\cref{lem:SDPtightness} implies that equality holds in \eqref{eq:qcqp_sdp} whenever $\cS(\set{M_1,\dots,M_m})$ is ROG and the SDP optimum value is bounded from below.
It may be natural to ask whether the boundedness assumption can be dropped in
the case where $B$ is specialized to $B=e_1e_1^\top$. Indeed, this is the only
case we need when analyzing \eqref{eq:qcqp_sdp}. The following example shows that this is not possible.
\begin{example}\label{ex:no_SDP_tightness_wo_boundedness}
Let $n = 2$ and $\cM = \set{\Sym(e_1e_2^\top), -\Sym(e_1e_2^\top)}$ so that
\begin{align*}
\cS(\cM) = \set{\begin{pmatrix}
	x_1^2 & 0\\0 & x_2^2
\end{pmatrix}:\, x\in\R^2} = \conv\left(\set{\begin{pmatrix}
	x_1\\0
\end{pmatrix}\begin{pmatrix}
	x_1\\0
\end{pmatrix}^\top:\, x_1\in\R}\cup\set{\begin{pmatrix}
	0\\x_2
\end{pmatrix}\begin{pmatrix}
	0\\x_2
\end{pmatrix}^\top:\, x_2\in\R}\right).
\end{align*}
The representation on the right shows that $\cS(\cM)$ is ROG. On the other hand, taking $B = e_1e_1^\top$ and $M_0 = -e_2e_2^\top$, we have
\begin{align*}
\inf_{x\in\R^2}\set{x^\top M_0 x:\, \begin{array}
	{l}
	xx^\top \in\cS(\cM)\\
	x^\top Bx = 1
\end{array}} &= \inf_{x\in\R^2}\set{-x_2^2:\, \begin{array}
	{l}
	x_1x_2 = 0\\
	x_1^2 = 1
\end{array}}= 0,
\end{align*}
which is not equal to
\begin{align*}
\inf_{X\in\S^2} \set{\ip{M_0, X} :\, \begin{array}
	{l}
	X \in\cS(\cM)\\
	\ip{B,X} = 1
\end{array}}
&= \inf_{x\in\R^2} \set{-x_2^2 :\, \begin{array}
	{l}
	x_1^2 = 1
\end{array}} = -\infty.\qedhere
\end{align*}
\end{example}

In a sense, \cref{ex:no_SDP_tightness_wo_boundedness} exhibits a particular worst-case behavior.
Specifically, adding an arbitrary inhomogeneous constraint to a ROG cone produces a set that is rank-two generated.
\begin{lemma}\label{lem:SDP_rank2_exactness}
Let $\cM\subseteq\S^n$. If $\cS(\cM)$ is ROG, then for all $B\in\S^n$,
\begin{align*}
\conv\left(\set{X\in\S^n:\, \begin{array}
	{l}
	\ip{M,X}\geq 0,\,\forall M\in\cM\\
	\ip{B,X} = 1\\
	X\succeq 0\\
	\rank(X) \leq 2
\end{array}}\right)
&=\set{X\in\S^n:\, \begin{array}
	{l}
	\ip{M,X}\geq 0,\,\forall M\in\cM\\
	\ip{B,X} = 1\\
	X\succeq 0
\end{array}}.
\end{align*}
In particular, when $\cS(\cM)$ is ROG, for any $M_0\in\S^n$, there exists a sequence of rank-two solutions approaching the SDP optimum value in \eqref{eq:qcqp_sdp}.
\end{lemma}
\begin{proof}
Let $\cL$ denote the inner set on the left hand side so that the left hand side is $\conv(\cL)$ and let $\cR$ denote the right hand set.

$(\subseteq)$ This follows upon noting that $\cL\subseteq\cR$ and $\cR$ is convex.

$(\supseteq)$ Let $X\in\cR$. As $\cR\subseteq\cS(\cM)$, we may decompose $X = \sum_{i=1}^r x_ix_i^\top$ where $x_ix_i^\top \in\cS(\cM)$ for all $i\in[r]$. We may assume that $r = \rank(X)$ by \cref{lem:caratheodory}.
Let $\beta_i\coloneqq \ip{B,x_ix_i^\top}$.

If $\beta_i>0$ for all $i\in[r]$, then we are done. Else, without loss of generality $\beta_1>0\geq \beta_2$.
Consider the value of $\mu\coloneqq\alpha_1\beta_1+\alpha_2\beta_2$ as $(\alpha_1,\alpha_2)$ moves continuously on the line segments $(1,0)\to(1,1)\to(0,1)$. Noting that $\beta_1>0$ and $\beta_2\leq 0$, we may fix $(\alpha_1,\alpha_2)$ on this path such that $\mu\in(0,1)$. Then, we can decompose
\begin{align*}
X = \mu\left(\frac{\alpha_1x_1x_1^\top + \alpha_2x_2x_2^\top}{\mu}\right) + (1-\mu)\left(\frac{X - \alpha_1x_1x_1^\top -\alpha_2x_2x_2^\top}{1-\mu}\right) \eqqcolon \mu X_{\ell} + (1-\mu) X_r.
\end{align*}
We have written $X$ as a convex combination of two matrices $X_{\ell}$ and $X_r$. It can be verified easily that $X_{\ell}\in\cL$ and $X_r\in\cR$. As at least one of $\alpha_1$ or $\alpha_2$ takes the value $1$, the element $X_r$ has rank strictly less than $r$. Iterating this procedure completes the proof.\qedhere
\end{proof}

\begin{remark}
A result similar to \cref{lem:SDP_rank2_exactness} in the case of a single \emph{homogeneous} constraint is presented in \cite[Lemma 5]{burer2015gentle}. Specifically, it is shown that for an arbitrary closed convex cone $\cS$, the extreme rays of the set obtained by intersecting  $\cS$ with a hyperplane through the origin can be written as convex combinations of at most two extreme rays of $\cS$.\qedhere
\end{remark}

\subsection{Convex hulls of bounded quadratically constrained sets}
Consider a set
\begin{align*}
\cY\coloneqq \set{y\in\R^{n-1}:\, q_i(y) \geq 0 ,\,\forall i\in[m]},
\end{align*}
where $q_i$s are quadratic functions of the form $q_i(y) = y^\top A_i y + 2b_i^\top y + c_i$. Let $M_i\coloneqq \left(\begin{smallmatrix}
	c_i & b_i^\top\\b_i & A_i
\end{smallmatrix}\right)$ and $\cM\coloneqq\set{M_1,\dots,M_m}$.

We begin by proving a technical lemma that will be useful in the remainder of this section. This lemma states that under a definiteness assumption, the set $\cY$, its projected SDP relaxation, and its SDP relaxation are each compact.
\begin{lemma}
\label{lem:cY_compact}
Suppose there exists $\lambda^*\in\R^m_+$ such that $\sum_{i=1}^m \lambda^* A_i$ is negative definite. Then, the following three sets are each compact:
\begin{align*}
&\cY \coloneqq \set{y\in\R^{n-1}:\, \begin{array}
	{l}
	y^\top A_i y + 2\ip{b_i, y} + c_i \geq 0 ,\,\forall i\in[m]
\end{array}},\\
&\set{y\in\R^{n-1}:\, \begin{array}
	{l}
	\exists Y\succeq yy^\top:\\
	\ip{A_i, Y} + 2\ip{b_i, y} + c_i \geq 0,\,\forall i\in[m]
\end{array}},\text{ and}\\
&\set{X\in\S^n_+:\, \begin{array}
	{l}
	\ip{M_i, X} \geq 0 ,\,\forall i\in[m]\\
	\ip{e_1e_1^\top, X} = 1
\end{array}}.
\end{align*}
\end{lemma}
\begin{proof}
For convenience, let $\cR_1$, $\cR_2$, $\cR_3$ denote the three sets in the lemma statement.
Let $A^* \coloneqq \sum_{i=1}^m \lambda^*_i A_i$. Similarly define $b^*$ and $c^*$. Note in particular that $A^*$ is negative definite.

To see that $\cR_3$ is compact, note that if $X\in\cR_3$, then for all $\mu\in\R$ we have 
\begin{align*}
\ip{\begin{pmatrix}
	c^* - \mu & (b^*)^\top\\
	b^* & A^*
\end{pmatrix}, X} \geq -\mu.
\end{align*}
By picking $\mu$ large enough, we can ensure that the matrix on the left hand side of this inequality is negative definite. We conclude that $\cR_3$ is bounded, whence compact.

Note that the $\cR_2$ is the image of the compact set $\cR_3$ under the continuous map $\left(\begin{smallmatrix}
	1 & y^\top\\
	y & Y
\end{smallmatrix}\right)\mapsto y$ so that $\cR_2$ is compact.

Finally, note that $\cR_1\subseteq\cR_2$ so that $\cR_1$ is bounded. As $\cR_1$ is closed, it is also compact.\qedhere
\end{proof}

The following lemma gives an explicit description of $\conv(\cY)$ under the
assumption that $\cS(\cM)$ is ROG and $\cY$ satisfies the above definiteness assumption.
\begin{proposition}
\label{prop:conv_hull_bounded_quadratic_set}
Suppose there exists $\lambda^*\in\R^m_+$ such that $\sum_{i=1}^m \lambda^*_i A_i$ is negative definite.
If $\cS(\cM)$ is ROG, then $\conv(\cY)$ is a semidefinite-representable set given by 
\begin{align}
\label{eq:conv_hull_bounded_quadratic_set}
\conv(\cY) = \set{y\in\R^{n-1}:\, \begin{array}
	{l}
	\exists Y \succeq yy^\top:\\
	\ip{A_i, Y} + 2\ip{b_i,y} + c_i \geq 0 ,\,\forall i\in[m]
\end{array}} .
\end{align}
\end{proposition}
\begin{proof}
As the assumptions of \cref{lem:cY_compact} hold, we have that both sides of \eqref{eq:conv_hull_bounded_quadratic_set} are compact. Therefore, it suffices to verify that the support function of $\cY$ and the support function of the set on the right hand side of \eqref{eq:conv_hull_bounded_quadratic_set} are equal.

Let $b_0\in\R^{n-1}$. Then,
\begin{align*}
\inf_{y\in\cY} \ip{b_0,y} &= 
{1\over 2}\inf_{x\in\R^n} \set{x^\top\begin{pmatrix}
	0 & b_0^\top\\
	b_0 & 0_{n-1}
\end{pmatrix}x:\, \begin{array}
	{l}
	x^\top M_i x\geq 0 ,\,\forall i\in[m]\\
	x^\top\left(e_1e_1^\top\right) x = 1
\end{array}}\\
&= {1\over 2} \inf_{X\in\S^n} \set{\ip{\begin{pmatrix}
	0 & b_0^\top\\
	b_0 & 0_{n-1}
\end{pmatrix},X}:\, \begin{array}
	{l}
	\ip{M_i, X}\geq 0 ,\,\forall i\in[m]\\
	\ip{e_1e_1^\top,X} = 1\\
	X\succeq 0
\end{array}}\\
&= \inf_{y\in\R^{n-1}}\set{\ip{b_0, y}:\, \begin{array}
	{l}
	\exists Y \succeq yy^\top:\\
	\ip{A_i, Y} + 2\ip{b_i,y} + c_i \geq 0 ,\,\forall i\in[m]
\end{array}}.
\end{align*}
Here, the first equality follows by writing $x = (1,\, y)$, the second equality follows by \cref{lem:SDPtightness}, and the third equality follows by writing $X = \left(\begin{smallmatrix}
	1 & y^\top\\ y & Y
\end{smallmatrix}\right)$.
\qedhere
\end{proof}

We next turn our attention to the closed convex hull of epigraph sets. Let $q_0$ be a quadratic function of the form $q_0(y) = y^\top A_0 y + 2b_0^\top y + c_0$ and define $M_0 \coloneqq \left(\begin{smallmatrix}
	c_0 & b_0^\top\\ b_0 & A_0
\end{smallmatrix}\right)$.
\begin{proposition}
\label{prop:epi_conv_hull_bounded_quadratic_set}
Suppose there exists $\lambda^*\in\R^m_+$ such that $A_0 - \sum_{i=1}^m \lambda^*_i A_i$ is positive definite.
If $\cS(\cM)$ is ROG, then the closed convex hull of
\begin{align*}
\epi \coloneqq \set{(y,t)\in\R^{n-1}\times \R:\, \begin{array}
	{l}
	q_0(y) \leq t\\
	y\in\cY
\end{array}}
\end{align*}
is a semidefinite-representable set given by
\begin{align*}
\clconv(\epi) = \set{(y,t)\in\R^{n-1}\times \R:\, \begin{array}
	{l}
	\exists Y \succeq yy^\top:\\
	\ip{A_0, Y} + 2\ip{b_0, y} + c_0 \leq t\\
	\ip{A_i, Y} + 2\ip{b_i,y} + c_i \geq 0 ,\,\forall i\in[m]
\end{array}}.
\end{align*}
\end{proposition}
\begin{proof}
Let $\cR$ denote the set on the right.

$(\subseteq)$ By taking $Y=yy^\top$, we have that $\epi\subseteq\cR$. It suffices to show that $\cR$ is both convex and closed. As $\cR$ is the projection of the SDP relaxation (a convex set) of $\epi$, it is itself convex. Next, consider a sequence $(y^{(i)},t^{(i)})\in\cR$ converging to $(y,t)$. Let $Y^{(i)}$ denote a sequence of matrices certifying $(y^{(i)},t^{(i)})\in\cR$. As there exists a $\lambda^*\in\R^m_+$ such that $A_0 - \sum_{i=1}^m \lambda^*_i A_i$ is positive definite, the sequence $Y^{(i)}$ is bounded and hence has a convergent subsequence with limit $Y$. By continuity, we deduce that $(y,t)\in\cR$ and hence $\cR$ is closed.

$(\supseteq)$ Suppose $(y,t)\notin\clconv(\epi)$. 
We will show that $(y,t)\notin\cR$.

First, we claim that $q_0(y)$ is bounded below on $\cY$.
Let $A^* \coloneqq A_0 - \sum_{i=1}^m \lambda^*_i A_i $ and similarly define $b^*$ and $c^*$. Then, for all $y\in\cY$, we have
\begin{align*}
q_0(y) &\geq q_0(y) - \sum_{i=1}^m \lambda^*_i q_i(y)=y^\top A^* y + 2\ip{b^*, y } + c^*\geq - (b^*)^\top (A^*)^{-1}b^* + c^*.
\end{align*}
We deduce that $q_0(y)$ is bounded below on $\cY$.

By the strict hyperplane separation theorem, there exists $(\mu,\nu)\neq(0,0)\in\R^{n-1}\times \R$ such that
\begin{align}
\label{eq:strict_separate_point_epi}
\ip{\mu,y} + \nu t &< \inf_{(y',t')\in\clconv(\epi)}\ip{\mu,y'} + \nu t' = \inf_{(y',t')\in\epi}\ip{\mu,y'} + \nu t'.
\end{align}
We claim that we may assume $\nu>0$ without loss of generality. First, suppose $\cY = \emptyset$. In this case, $\epi=\emptyset$ and any arbitrary $(\mu,\nu)\neq(0,0)$ satisfies \eqref{eq:strict_separate_point_epi}. On the other hand, if $\cY$ is nonempty then $e_n$ is a recessive direction for $\epi$. In particular, as the objective value of the program on the right is finite (by the bound on the left), we deduce that $\nu\geq 0$. Finally, as $q_0(y)$ is bounded below on $\cY$, we may increase $\nu$ by some positive amount without affecting \eqref{eq:strict_separate_point_epi}.

Then,
\begin{align*}
\ip{\mu,y} + \nu t &< \min_{y'}\set{\ip{\mu,y'} + \nu q_0(y'):\, y'\in\cY}\\
&= \min_{y',Y'}\set{\ip{\mu,y'} + \nu(\ip{A_0, Y'} + 2\ip{b_0, y'} + c_0):\, \begin{array}
	{l}
	Y'\succeq y'y'^{\top}\\
	\ip{A_i, Y'} + 2\ip{b_i,y'} + c_i \geq 0 ,\,\forall i\in[m]
\end{array}}\\
&\leq \min_{Y}\set{\ip{\mu,y} + \nu(\ip{A_0, Y} + 2\ip{b_0, y} + c_0):\, \begin{array}
	{l}
	Y\succeq yy^\top\\
	\ip{A_i, Y} + 2\ip{b_i,y} + c_i \geq 0 ,\,\forall i\in[m]
\end{array}}.
\end{align*}
Here, the first line follows by substituting the optimal value of $t'$ in \eqref{eq:strict_separate_point_epi}, the second line follows from \cref{lem:SDPtightness} (which we can apply as $\cS(\cM)$ is ROG and the SDP on the second line has finite objective value), and the third line follows by selecting $y' = y$.

Subtracting $\ip{\mu,y}$ from both sides and dividing by $\nu>0$ leads to the desired conclusion that $(y,t)\notin\cR$ and completes the proof.\qedhere
\end{proof}

Applying a perturbation argument to
\cref{prop:epi_conv_hull_bounded_quadratic_set} allows us to additionally relax
the assumption that $A_0 - \sum_{i=1}^m \lambda^*_i A_i$ is positive definite.

\begin{corollary}
  \label{cor:epi_conv_hull_bounded_quadratic_set_psd}
  Suppose there exists $\lambda^*\in\R^m_+$ such that $A_0 - \sum_{i=1}^m \lambda^*_i A_i$ is positive semidefinite.
  If $\cS(\cM)$ is ROG, then the closed convex hull of
  \begin{align*}
    \epi \coloneqq \set{(y,t)\in\R^{n-1}\times \R:\, \begin{array}
                                                       {l}
                                                       q_0(y) \leq t\\
                                                       y\in\cY
                                                     \end{array}}
  \end{align*}
  is the closure of a semidefinite-representable set:
  \begin{align*}
    \clconv(\epi) = \cl\left(\set{(y,t)\in\R^{n-1}\times \R:\, \begin{array}
                                                        {l}
                                                        \exists Y \succeq yy^\top:\\
                                                        \ip{A_0, Y} + 2\ip{b_0, y} + c_0 \leq t\\
                                                        \ip{A_i, Y} + 2\ip{b_i,y} + c_i \geq 0 ,\,\forall i\in[m]
                                                      \end{array}}\right).
  \end{align*}
\end{corollary}
\begin{proof}
Let $\cR$ denote the set inside the right hand side so that the desired
conclusion is $\clconv(\epi) = \overline{\cR}$.

$(\subseteq)$ This direction follows simply from observing that $\epi\subseteq
\cR$ and that $\cR$ is convex.

$(\supseteq)$ Let $(\hat y,\hat t)\in \cR$ and let $\hat Y$ be a matrix certifying
$(\hat y, \hat t)\in\cR$.
It suffices to show that
$(\hat y,\hat t+\epsilon)\in\clconv(\epi)$ for all $\epsilon>0$. Let $A_0' \coloneqq A_0 +\delta
I$ where we have set $\delta\coloneqq \epsilon/\tr(\hat{Y})$.
Define $q_0'(y) \coloneqq q_0(y) + \delta \norm{y}^2 = y^\top A_0' y + 2\ip{b_0, y} + c_0$.
Note that by construction,
\begin{align*}
\ip{A_0', \hat Y} + 2\ip{b_0, \hat y} + c_0 = \left( \ip{A_0, \hat Y} + 2\ip{b_0, \hat y} + c_0 \right) + \epsilon \leq \hat t +\epsilon
\end{align*}
so that
\begin{align*}
(\hat y, \hat t+\epsilon) \in \set{(y, t)\in\R^{n-1}\times \R:\ \begin{array}
                                                                  {l}
                                                                  \exists Y\succeq yy^\top:\\
                                                                  \ip{A_0', Y} + 2\ip{b_0, y} + c_0 \leq t\\
                                                                  \ip{A_i, Y} + 2\ip{b_i, y} + c_i \geq 0 ,\,\forall i\in[m]
                                                                \end{array}}.
\end{align*}

Next, as $\cS(\cM)$ is ROG and $A_0' - \sum_{i=1}^m \lambda_i^* A_i = \left( A_0 -\sum_{i=1}^m \lambda_i^* A_i \right) + \delta I$ is positive definite,
we may apply \cref{prop:epi_conv_hull_bounded_quadratic_set} with
$q_0'(y)$ to deduce that
\begin{align*}
(\hat y, \hat t+\epsilon) &\in \clconv\left( \set{(y,t): \begin{array}
                                                          {l}
                                                          q_0'(y) \leq t\\
                                                          y\in\cY
                                                        \end{array}} \right)\\
  &\subseteq \clconv(\epi).
\end{align*}
Here, the second containment follows by noting that $q_0(y) \leq q_0'(y)$ for
all $y$.\qedhere
\end{proof}

The following example shows how to recover \cite[Theorem 4]{wang2020generalized} as an immediate corollary of \cref{lem:M_leq_1,cor:epi_conv_hull_bounded_quadratic_set_psd}.
\begin{example}
Consider a set $\cY$ defined by a single quadratic inequality constraint
\begin{align*}
\cY = \set{y\in\R^{n-1}:\, q_1(y)\geq 0}.
\end{align*}
The associated cone $\cS(\set{M_1})$ is ROG by \cref{lem:M_leq_1}. Next, suppose $q_0(x)$ is a quadratic objective function for which there exists $\lambda\geq 0$ such that $A_0 - \lambda A_1\succeq 0$. Then, \cref{cor:epi_conv_hull_bounded_quadratic_set_psd} implies that
\begin{align*}
\clconv\left(\set{(y,t)\in\R^{n-1}\times\R:\, \begin{array}
	{l}
	q_0(y) \leq t\\
	q_1(y) \geq 0
\end{array}}\right) &=
\cl\left(\set{(y,t)\in\R^{n-1}\times \R:\, \begin{array}
	{l}
	\exists Y\succeq yy^\top\,:\\
	\ip{A_0, Y} + 2\ip{b_0, y} + c_0 \leq t\\
	\ip{A_1, Y} + 2\ip{b_1, y} + c_1 \geq 0
\end{array}}\right).\qedhere
\end{align*}
\end{example}

We next examine a classical example related to the ``perspective reformulation/relaxation'' trick~\cite{gunluk2010perspective,ceria1999convex,frangioni2006perspective}
and demonstrate how this convex hull result can be recovered using our ROG toolsets.
The nonconvex set in this example will involve both binary and continuous variables and complementarity constraints.
\begin{example}
\label{ex:mixed_binary_set}
Consider the quadratically constrained set
\begin{align*}
\cY = \set{y\in\R^2:\, \begin{array}
	{l}
	(1-y_1)y_1 = 0\\
	(1-y_1)y_2 = 0
\end{array}}.
\end{align*}
In words, $y_1$ is constrained to be a binary variable, $y_2$ is allowed to be
arbitrary when $y_1=1$ is ``on'' and forced to be zero when $y_1=0$ is ``off.''

Letting
$M_1\coloneqq \Sym((e_3 - e_1)e_1^\top)$ and $M_2\coloneqq\Sym((e_3-e_1)e_2^\top)$, we have that
\begin{align*}
\cY = \set{y\in\R^2:\, \begin{array}
	{l}
	\begin{pmatrix}y\\1\end{pmatrix}^\top
	M_1
	\begin{pmatrix}y\\1\end{pmatrix} = 0\\
	\begin{pmatrix}y\\1\end{pmatrix}^\top
	M_2
	\begin{pmatrix}y\\1\end{pmatrix} = 0\\
\end{array}}.
\end{align*}
Let $\cM = \set{M_1,M_2}$
and note that $\cT(\cM)$ is ROG by \cref{cor:ac_bc_suffices}.

Next, we rewrite $\cY$ using inequality constraints so that we may apply
\cref{prop:epi_conv_hull_bounded_quadratic_set}.
Letting $q_1(y) = (1-y_1)y_1$, $q_2(y)=-(1-y_1)y_1$, $q_3(y)=(1-y_1)y_2$,
and $q_4(y)=-(1-y_1)y_2$, we may write
\begin{align*}
  \cY = \set{y\in\R^2:\, \begin{array}
                           {l}
                           q_i(y)\leq 0,\,\forall i \in[4]
                         \end{array}}.
\end{align*}
Note that $A_1 = -e_1e_1^\top$, $A_2 = e_1e_1^\top$, $A_3 = -\Sym(e_1e_2^\top)$,
and $A_4 = \Sym(e_1e_2^\top)$.
Setting $q_0(y) = y_2^2$, we have that $A_0 = e_2e_2^\top$.
Then, as $A_0 + A_2\succ 0$, we deduce that the assumptions
of \cref{prop:epi_conv_hull_bounded_quadratic_set} hold.
Applying \cref{prop:epi_conv_hull_bounded_quadratic_set} then gives
\begin{align*}
\clconv\set{(y,t)\in\R^2\times \R:\, \begin{array}
	{l}
	y_2^2 \leq t\\
	(1-y_1)y_1 = 0\\
	(1-y_1)y_2 = 0
\end{array}}
&=\set{(y,t)\in\R^2\times\R:\, \begin{array}
	{l}
	\exists Y\succeq yy^\top\\
	Y_{2,2}\leq t\\
	y_1 - Y_{1,1} = 0\\
	y_2 - Y_{1,2} = 0
\end{array}}\\
&= \set{(y,t)\in\R^2\times\R:\, \begin{array}
	{l}
	\begin{pmatrix}
	y_1 & y_2\\ y_2 & t
	\end{pmatrix}\succeq yy^\top
\end{array}}\\
&= \set{(y,t)\in\R^2\times\R:\, \begin{array}
	{l}
	y_1\geq y_1^2\\
	t\geq y_2^2\\
	(y_1-y_1^2) (t-y_2^2) \geq (y_2-y_1y_2)^2
\end{array}}
.
\end{align*}
Note that the first constraint in the last formulation implies that $y_1\in[0,1]$. By expanding and rearranging, we can write the last constraint as
\begin{align*}
0 &\leq (y_1-y_1^2) (t-y_2^2) - (y_2-y_1y_2)^2 
= y_1 t + y_1 y_2^2 - y_1^2 t - y_2^2 = (y_1t-y_2^2)(1-y_1).
\end{align*} 
When $y_1\in[0,1)$, this constraint is equivalent to $y_1t-y_2^2\geq 0$.
On the other hand when $y_1=1$, the constraint $y_1t-y_2^2\geq 0$ is redundant. Hence, we deduce that
\begin{align*}
\clconv\set{(y,t)\in\R^2\times \R:\, \begin{array}
	{l}
	y_2^2 \leq t\\
	(1-y_1)y_1 = 0\\
	(1-y_1)y_2 = 0
\end{array}}
&=\set{(y,t)\in\R^2\times\R:\, \begin{array}
	{l}
	y_1\in[0,1]\\
	y_1t \geq y_2^2
\end{array}}
.
\end{align*}
This gives the well-known perspective formulation of $\clconv(\cY)$.\qedhere
\end{example}

\begin{remark}
There are few known sufficient conditions guaranteeing that the convex hull of the epigraph of a QCQP is given by its SDP relaxation.
The conditions presented by \citet[Theorems 1 and 7]{wang2019tightness} are among the most general in this direction.
We claim that both \cite[Theorems 1 and 7]{wang2019tightness} are incomparable with \cref{prop:epi_conv_hull_bounded_quadratic_set}.
Note that \cite[Theorem 1]{wang2019tightness} cannot be applied directly to \cref{ex:mixed_binary_set}: the set of \textit{convex Lagrange multipliers} (see \cite[Section 2.1]{wang2019tightness}) for this example is
\begin{align*}
\Gamma&\coloneqq \set{\gamma\in\R^2:\, \begin{pmatrix}
0 & \\ & 1	
\end{pmatrix} + \gamma_1 \begin{pmatrix}
-1 &\\&0
\end{pmatrix} + \gamma_2 \begin{pmatrix}
0 & -1/2\\
-1/2 & 0
\end{pmatrix}\succeq 0}\\
&= \set{\gamma\in\R^2:\, \gamma_1\leq 0,\, \abs{\gamma_2}\leq \sqrt{-\gamma_1}},
\end{align*}
which is not polyhedral.
On the other hand, \cite[Theorem 1]{wang2019tightness} can be applied to QCQPs where the $A_i$s satisfy a ``symmetry'' condition. The following QCQP is such an example. Consider
\begin{align*}
\inf_{y\in\R^4}\set{\norm{y}^2:\, \begin{array}
	{l}
	y^\top \left(\begin{smallmatrix}
	1 &&&\\
	& 1 &&\\
	&& -1 &\\
	&&& -1
\end{smallmatrix}\right)y + 1\geq 0\\
y^\top \left(\begin{smallmatrix}
	-2 &&&\\
	& -2 &&\\
	&& 1 &\\
	&&& 1
\end{smallmatrix}\right)y + 1\geq 0
\end{array}}.
\end{align*}
The corresponding set $\cM$ for this example is $\cM = \set{\Diag(1,1,-1,-1,1),\Diag(-2,-2,1,1,1)}$.
\cref{thm:two_LMI_NS} implies that $\cS(\cM)$ is not ROG and thus \cref{prop:epi_conv_hull_bounded_quadratic_set} cannot be applied to this example.
We conclude that \cite[Theorem 1]{wang2019tightness} and \cref{prop:epi_conv_hull_bounded_quadratic_set} are incomparable. Similar examples can be constructed to show that \cite[Theorem 7]{wang2019tightness} and \cref{prop:epi_conv_hull_bounded_quadratic_set} are incomparable.\qedhere
\end{remark}

\section*{Acknowledgments}
The authors wish to thank the review team for their constructive feedback that improved the presentation of the material in this paper. 
This research is supported in part by NSF grant CMMI 1454548 and ONR grant N00014-19-1-2321. 
Part of this work was done while the second author was visiting the Simons Institute for the Theory of Computing. It was partially supported by the DIMACS/Simons Collaboration on Bridging Continuous and Discrete Optimization through NSF grant CCF-1740425.

{
\bibliographystyle{plainnat} 
\bibliography{ROG_bib_clean.bib}
}

\appendix
\section{Proof of \cref{lem:N_at_most_four_lines}}
\label{ap:bezout_proof}
For completeness we restate \cref{lem:N_at_most_four_lines}.

\bezouttheorem*
\begin{proof}
As $\alpha_1 M_1 + \alpha_2 M_2 \notin \S^3_+$ for any $(\alpha_1, \alpha_2)\neq(0,0)$, we
have that $M_1$ and $M_2$ must each have rank either two or three.
We will break the proof into two cases.

Suppose first that $\rank(M_1)=\rank(M_2)=2$.
As $M_1,M_2\notin\S^3_+$, each $M_i$ has exactly one positive and one negative eigenvalue. We can then write $M_1 = \Sym(ab^\top)$ and $M_2 = \Sym(cd^\top)$.
Then
\begin{align*}
\cN(\cM) &= \set{x:\, x^\top(ab^\top)x = x^\top(cd^\top)x = 0}\\
&= (a^\perp \cup b^\perp)\cap (c^\perp \cup d^\perp)\\
&= (a^\perp \cap c^\perp) \cup (a^\perp \cap d^\perp)  \cup(b^\perp \cap c^\perp) \cup (b^\perp \cap d^\perp).
\end{align*}
As condition (ii) does not hold, each of the four spaces on the final line have dimension one. Thus $\cN(\cM)$ is the union of at most four distinct lines.

Next suppose without loss of generality that $\rank(M_1)= 3$. As $M_1\notin\S^3_+$, we may assume that it has two positive eigenvalues and one negative eigenvalue. Performing a change of basis, it suffices to consider when
\begin{align*}
M_1 = \begin{pmatrix}
	1 &&\\ &1& \\ &&-1
\end{pmatrix}
\qquad\text{and}\qquad
M_2 = \begin{pmatrix}
	a& b& c\\ b& d& e\\ c&e&f
\end{pmatrix}.
\end{align*}
We will consider the intersection $\cN(\cM)\cap\set{x\in\R^3:\, x_3 = 1}$.
Note that if $x\in\cN(\cM)$ has $x_3$ coordinate equal to zero, then $x = 0$.
Thus, the number of distinct lines in $\cN(\cM)$ is equal to the number of distinct points in
\begin{align*}
\cP \coloneqq \set{(x_1,x_2)\in\R^2:\, \begin{array}
	{l}
	x_1^2+x_2^2 -1 = 0\\
	\left(ax_1^2 +dx_2^2 + 2cx_1 + f\right) + x_2 \left(2bx_1 + 2e\right) = 0
\end{array}}.
\end{align*}
Suppose that $\cN(\cM)$ contains at least five lines so that $\cP$ contains at least five points. Without loss of generality, we may assume that the $x_1$ coordinates of these five points are distinct (else, perform an orthonormal change of basis on the first two dimensions). Let the $x_1$ coordinates of these five points be $\xi_1,\xi_2,\dots,\xi_5$.
For each $\xi_i$, by the first constraint in the definition of $\cP$, we have that the corresponding $x_2$ coordinate must be either $\sqrt{1-\xi_i^2}$ or $-\sqrt{1-\xi_i^2}$. Hence, 
\begin{align*}
&\left[\left(a\xi^2 +d(1-\xi^2) + 2c\xi + f\right) + \sqrt{1-\xi^2} \left(2b\xi + 2e\right)\right]\left[\left(a\xi^2 +d(1-\xi^2) + 2c\xi + f\right) - \sqrt{1-\xi^2} \left(2b\xi + 2e\right)\right]\\
&\quad =
\left[(a-d)^2+4b^2\right] \xi^4 +
\left[4(a-d)c +8be\right] \xi^3 +
\left[2(a-d)(d+f)+4c^2+4e^2-4b^2\right] \xi^2 +\\
&\qquad\quad
\left[4c(d+f)-8be\right] \xi +
\left[(d+f)^2 - 4e^2\right]
\end{align*}
is a degree-4 polynomial in $\xi$ which is zero on five distinct points $\xi_1,\dots,\xi_5$. We conclude that this polynomial is identically zero. The coefficient of $\xi^4$ implies that $a=d$ and $b = 0$. The coefficient of $\xi^2$ implies that $c = e= 0$. The constant term implies that $f=-d$. We conclude that $M_2$ has the form
\begin{align*}
M_2 = \begin{pmatrix}
	a && \\ &a&\\ &&-a
\end{pmatrix}.
\end{align*}
This contradicts the assumption that there does not exist an $(\alpha_1,\alpha_2)\neq(0,0)$ such that $\alpha_1M_1 + \alpha_2M_2 \in\S^n_+$.\qedhere
\end{proof}

\end{document}